\documentclass[10pt]{amsart}

\usepackage{color}

\usepackage{amsmath}
\usepackage{amsthm}
\usepackage{amssymb}
\usepackage{amscd}
\usepackage{lscape}
\usepackage{graphicx}

\usepackage[all,cmtip]{xy}
\usepackage{sseq}
\xyoption{dvips}

\def \Ochanine{MR895567}
\def \Weibel{MR1269324}
\def \GreenleesMay{MR1230773}
\def \KatzMazur{MR772569}
\def \Katz{MR0447119}
\def \HartshorneAG{MR0463157}
\def \GHMR{MR2183282}
\def \DeligneRapoport{MR0337993}
\def \BrownComenetz{MR0405403}
\def \HopkinsSinger{MR2192936}

\def \HoveyStrickland{?}
\def \Silverman{MR817210}
\def \Bauer{MR2508200}
\def \Rezk512{R1}
\def \Behrens{MR2193339}
\def \LewisMaySteinberger{MR866482}

\def \MahowaldRezk{MR1719751}

\def \HoveyStrickland{MR1601906}

\def \FauskHuMay{MR1988072}
\def \Deligne{MR0498551}

\def \HKM{HKM}

\def \JardineSpectra{MR1765868}
\def \JardineSimplicial{MR906403}
\def \DHI{MR2034012}
\def \RavenelLocalization{MR737778}
\def \Champs{MR1771927}
\def \ConstructTMF{ConstructTMF}
\def \GH{GH}
\def \GoerssHopkins{MR2125040}
\def \Cooke{MR0461544}
\def \Hollander{MR2391126}

\def\Topos{MR2522659}
\def\Algebra{HigherAlgebra}
\def\DAG{DAG}

\def\HAGI{MR2137288}
\def\HAGII{MR2394633}

\def \M{\mathcal{M}}

\def \Mc2{\tilde{M}(2)}

\def \Msing{\mathcal{M}^{\infty}}

\def \Z3{\mathbb{Z}_{(3)}}
\def \Z{\mathbb{Z}}
\def \Q{\mathbb{Q}}
\def \F{\mathbb{F}}

\def \P {\mathbb{P}}
\def \A {\mathbb{A}}

\def \G {\mathbb{G}}

\def \tH {\hat{H}}

\def\IZ{I_{\Z}}
\def\IQ{I_{\Q}}
\def\IQZ{I_{\Q/\Z}}
\def\IR{I_{R}}

\def\Gl{GL_2(\Z/n)}

\DeclareMathOperator{\Proj}{Proj}
\DeclareMathOperator{\Spec}{Spec}
\DeclareMathOperator{\Spf}{Spf}

\DeclareMathOperator{\Hom}{Hom}
\DeclareMathOperator{\Ext}{Ext}

\DeclareMathOperator{\Id}{Id}

\DeclareMathOperator{\sgn}{sgn}

\DeclareMathOperator{\Perm}{Perm}
\DeclareMathOperator{\Ker}{Ker}
\DeclareMathOperator{\Coker}{Coker}

\DeclareMathOperator{\Tot}{Tot}
\DeclareMathOperator{\holim}{holim}
\DeclareMathOperator{\Rings}{(E_\infty-Rings)}
\DeclareMathOperator{\Tel}{Tel}

\newcommand{\SSh}[1]{\mathcal{O}_{#1}}
\newcommand{\Sh}[1]{\mathcal{#1}}

\newcommand{\om}[2]{\omega_{#1}^{{#2}}}
\newcommand{\Om}[1]{\Omega_{#1}}

\newcommand{\Ml}[1]{\mathcal{M}{(#1)}}

\newcommand{\Mlsing}[1]{\mathcal{M}{(#1)}^{\infty}}

\newcommand{\cross}[1]{\underset{#1}{\times}}
\newcommand{\tensor}[1]{\underset{#1}{\otimes}}

\newcommand{\smsh}[1]{\underset{#1}{\wedge}}

\newcommand{\Cat}[1]{\mathcal{#1}}

\newtheorem{thm}{Theorem}[section]
\newtheorem{cor}[thm]{Corollary}
\newtheorem{lemma}[thm]{Lemma}

\newtheorem{proposition}[thm]{Proposition}
\newtheorem*{thmM}{Theorem \ref{MainThm}}

\theoremstyle{remark}
\newtheorem{rem}[thm]{Remark}

\theoremstyle{remark}

\theoremstyle{remark}

\theoremstyle{definition}
\newtheorem{definition}[thm]{Definition}

\theoremstyle{remark}

{\unskip\nobreak\hskip 1em plus 1fil\nobreak$\Box$
\parfillskip=0pt%
\endtrivlist}

\begin{document}
\title{Duality for Topological Modular Forms}
\author{Vesna Stojanoska}

\address{
Vesna Stojanoska,
Massachusetts Institute of Technology,
Cambridge MA 02139}

\email{vstojanoska@math.mit.edu}

\begin{abstract}
It has been observed that certain localizations of the spectrum of topological modular forms are self-dual (Mahowald-Rezk, Gross-Hopkins). We provide an integral explanation of these results that is internal to the geometry of the (compactified) moduli stack of elliptic curves $ \M $, yet is only true in the derived setting. When $ 2 $ is inverted, a choice of level $ 2 $ structure for an elliptic curve provides a geometrically well-behaved cover of $ \M $, which allows one to consider $ Tmf $ as the homotopy fixed points of $ Tmf(2) $, topological modular forms with level $ 2 $ structure, under a natural action by $ GL_2(\Z/2) $. As a result of Grothendieck-Serre duality, we obtain that $ Tmf(2) $ is self-dual. The vanishing of the associated Tate spectrum then makes $ Tmf $ itself Anderson self-dual.
\end{abstract}

\maketitle

\section{Introduction}

There are several notions of duality in homotopy theory, and a family of such originates with Brown and Comenetz. In \cite{\BrownComenetz}, they introduced the spectrum $ \IQZ $ which represents the cohomology theory that assigns to a spectrum $ X $ the Pontryagin duals $ \Hom_{\Z}(\pi_{-*}X,\Q/\Z) $ of its homotopy groups. Brown-Comenetz duality is difficult to tackle directly, so Mahowald and Rezk \cite{\MahowaldRezk} studied a tower approximating it, such that at each stage the self-dual objects turn out to have interesting chromatic properties. In particular, they show that self-duality is detected on cohomology as a module over the Steenrod algebra. Consequently, a version of the spectrum $ tmf $ of topological modular forms possesses self-dual properties.

The question thus arises whether this self-duality of $ tmf $ is already inherent in its geometric construction. Indeed, the advent of derived algebraic geometry not only allows for the construction of an object such as $ tmf $, but also for bringing in geometric notions of duality to homotopy theory, most notably Grothendieck-Serre duality. However, it is rarely possible to identify the abstractly constructed (derived) dualizing sheaves with a concrete and computable object. This stands in contrast to ordinary algebraic geometry, where a few smallness assumptions guarantee that the sheaf of differentials is a dualizing sheaf (eg. \cite[III.7]{\HartshorneAG}).

Nevertheless, the case of the moduli stack of elliptic curves\footnote{The moduli stack of elliptic curves is sometimes also denoted $ \M_{ell} $ or $ \M_{1,1} $; we chose the notation from \cite{\DeligneRapoport}.} $ \M^0 $ and topological modular forms allows for a hands-on approach to Serre duality in a derived setting. Even if the underlying ordinary stack does not admit a traditional Serre duality pairing, the derived version of $ \M^0 $ is considerably better behaved. The purpose of this paper is to show how the input from homotopy theory brings duality back into the picture. Conversely, it provides an integral interpretation of the aforementioned self-duality for $ tmf $ that is inherent in the geometry of the moduli stack of elliptic curves.

The duality that naturally occurs from the geometric viewpoint is not quite that of Brown and Comenetz, but an integral version thereof, called Anderson duality and denoted $ \IZ $ \cite{A1,\HopkinsSinger}. After localization with respect to Morava $ K $-theory $ K(n) $ for $ n>0 $, however, Anderson duality and Brown-Comenetz duality only differ by a shift.

Elliptic curves have come into homotopy theory because they give rise to interesting one-parameter formal groups of heights one or two. The homotopical version of these is the notion of an elliptic spectrum: an even periodic spectrum $ E $, together with an elliptic curve $ C $ over $ \pi_0 E $, and an isomorphism between the formal group of $ E $ and the completion of $ C $ at the identity section. \'Etale maps $ \Spec \pi_0 E\to \M^0 $ give rise to such spectra; more strongly, as a consequence of the Goerss-Hopkins-Miller theorem, the assignment of an elliptic spectrum to an \'etale map to $ \M^0 $ gives an \'etale sheaf of $ E_\infty $-ring spectra on the moduli stack of elliptic curves. Better still, the compactification of $ \M^0 $, which we will hereby denote by $ \M $, admits such a sheaf, denoted $ \Sh{O}^{top} $, whose underlying ordinary stack is the usual stack of generalized elliptic curves \cite{\ConstructTMF}. The derived global sections of $ \Sh{O}^{top} $ are called $ Tmf $, the spectrum of topological modular forms. This is the non-connective, non-periodic version of $ Tmf $.\footnote{We chose this notation to distinguish this version of topological modular forms from the connective $ tmf $ and the periodic $ TMF $.}

The main result proved in this paper is the following theorem:

\begin{thmM}
The Anderson dual of $ Tmf[1/2] $ is $ \Sigma^{21}Tmf[1/2] $.
\end{thmM}
The proof is geometric in the sense that it uses Serre duality on a cover of $ \M $ as well as descent, manifested in the vanishing of a certain Tate spectrum.

\subsection*{Acknowledgments}

This paper is part of my doctoral dissertation at Northwestern University, supervised by Paul Goerss. His constant support and guidance throughout my doctoral studies and beyond have been immensely helpful at every stage of this work. I am indebted to Mark Behrens, Mark Mahowald, and Charles Rezk for very helpful conversations, suggestions, and inspiration. I am also grateful to the referee who has made extraordinarily thorough suggestions to improve an earlier version of this paper, and to Mark Behrens and Paul Goerss again for their patience and help with the material in Section \ref{sec:dermoduli}.

Many thanks to Tilman Bauer for creating the sseq latex package, which I used to draw the spectral sequences in Figures \ref{Fig:tmf2} through \ref{fig:htpyorbit}.

\section{Dualities}\label{sec:dualities}

We begin by recalling the definitions and properties of Brown-Comenetz and Anderson duality.

In \cite{\BrownComenetz}, Brown and Comenetz studied the functor on spectra
\[ X\mapsto I_{\Q/\Z}^*(X)=\Hom_{\Z}(\pi_{-*}X,\Q/\Z). \]
Because $ \Q/\Z $ is an injective $ \Z $-module, this defines a cohomology theory, and is therefore represented by some spectrum; denote it $ \IQZ $. We shall abuse notation, and write  $ \IQZ $ also for the functor
\[ X\mapsto F(X,\IQZ), \]
which we think of as ``dualizing" with respect to $ \IQZ $. And indeed, if $ X $ is a spectrum whose homotopy groups are finite, then the natural ``double-duality" map $ X\to \IQZ \IQZ X $ is an equivalence. 

In a similar fashion, one can define $ \IQ $\footnote{In fact, $ \IQ $ is the Eilenberg-MacLane spectrum $ H\Q $.} to be the spectrum corepresenting the functor \[ X\mapsto \Hom_{\Z}(\pi_{-*}X,\Q). \] The quotient map $ \Q\to\Q/\Z $ gives rise to a natural map $ \IQ\to\IQZ $; in accordance with the algebraic situation, we denote by $ \IZ $ the fiber of the latter map. As I have learned from Mark Behrens, the functor corepresented by $ \IZ $ has been introduced by Anderson in \cite{A1}, and further employed by Hopkins and Singer in \cite{\HopkinsSinger}.

For $ R $ any of $ \Z $, $ \Q $, or $ \Q/\Z $, denote also by $ I_R $ the functor on spectra $ X\mapsto F(X,I_R) $. When $ R $ is an injective $ \Z $-module, we have that $ \pi_*I_RX=\Hom_{\Z}(\pi_{-*}X,R) $.

Now for any spectrum $ X $, we have a fiber sequence
\begin{align}\label{IZresolution}
\IZ X\to\IQ X\xrightarrow{\varphi} \IQZ X,
\end{align}
giving a long exact sequence of homotopy groups
\[
\xymatrix@C=-5pt{
\cdots && \pi_{t+1}\IQZ X \ar[rr] \ar[rd] && \pi_t \IZ X \ar[rr]\ar[rd] && \pi_t\IQ X &&\cdots\\
&&& \Coker \pi_{t+1}\varphi \ar[ur] && \Ker \pi_t \varphi\ar[ur]\\
}\]
But the kernel and cokernel of $ \pi_*\varphi $ compute the derived $ \Z $-duals of the homotopy groups of $ X $, so we obtain short exact sequences \[ 0\to \Ext^1_{\Z}(\pi_{-t-1}X,\Z)\to \pi_t\IZ X \to \Hom_{\Z}(\pi_{-t}X,\Z)\to 0.\] We can think of these as organized in a two-line spectral sequence:
\begin{align}\label{ss:AndersonDual}
\Ext^s_{\Z}(\pi_{t}X,\Z)\Rightarrow \pi_{-t-s}\IZ X.
\end{align}

\begin{rem}\label{rem:Anderson}
In this project we will use Anderson duality in the localization of the category of spectra where we invert some integer $ n $ (we will mostly be interested in the case $ n=2 $). The correct way to proceed is to replace $ \Z $ by $ \Z[1/n] $ everywhere. In particular, the homotopy groups of the Anderson dual of $ X $ will be computed by $ \Ext^s_{\Z[1/n]}(\pi_{t}X,\Z[1/n]) $.
\end{rem}

\subsection{Relation to $ K(n) $-local and Mahowald-Rezk Duality}

This section is a digression meant to suggest the relevance of Anderson duality by relating it to $ K(n) $-local Brown-Comenetz dualty as well as Mahowald-Rezk duality.

We recall the definition of the Brown-Comenetz duality functor in the $ K(n) $-local category from \cite[Ch. 10]{\HoveyStrickland}. Fix a prime $ p $, and for any $ n\geq 0$, let $ K(n) $ denote the Morava $ K $-theory spectrum, and let $ L_n $ denote the Bousfield localization functor with respect to the Johnson-Wilson spectrum $ E(n) $. There is a natural transformation $ L_n\to L_{n-1} $, whose homotopy fiber is called the monochromatic layer functor $ M_n $, and a homotopy pull-back square
\begin{align}\label{fracturesquare}
\xymatrix{ L_n \ar[r]\ar[d] & L_{K(n)}\ar[d] \\
L_{n-1} \ar[r] & L_{n-1}L_{K(n)},
}
\end{align}
implying that $ M_n $ is also the fiber of $ L_{K(n)} \to L_{n-1}L_{K(n)}$.

The $ K(n) $-local Brown-Comenetz dual of $ X $ is defined to be $ I_nX=F(M_nX,\IQZ) $, the Brown-Comenetz dual of the $ n $-th monochromatic layer of $ X $. By \eqref{fracturesquare}, $ I_nX $ only depends on the $ K(n) $-localization of $ X $ (since $ M_nX $ only depends on $ L_{K(n)}X $), and by the first part of the proof of Proposition \ref{prop:local} below, $ I_nX $ is $ K(n) $-local. Therefore we can view $ I_n $ as an endofunctor of the $ K(n) $-local category.

Note that after localization with respect to the Morava $ K $-theories $ K(n) $ for $ n\geq 1 $, $ \IQ X $ becomes contractible since it is rational. The fiber sequence \eqref{IZresolution} then gives that $ L_{K(n)}\IQZ X= L_{K(n)}\Sigma\IZ X $. By Proposition \ref{prop:local} below, we have that $ I_nX $ is the $ K(n) $-localization of the Brown-Comenetz dual of $ L_nX $. In particular, if $ X $ is already $ E(n) $-local, then
\begin{align}\label{AndersonGrossHopkins}
I_nX= L_{K(n)}\Sigma\IZ X.
\end{align}

In order to define the Mahowald-Rezk duality functor \cite{\MahowaldRezk} we need some preliminary definitions. Let $ T_i $ be a sequence of finite spectra of type $ i $ and let $ \Tel(i) $ be the mapping telescope of some $ v_i $ self-map of $ T_i $. The finite localization $ L_n^f $ of $ X $ is the Bousfield localization with respect to the wedge $ \Tel(0)\vee\cdots \vee \Tel(n) $, and  $ C_n^f $ is the fiber of the localization map $ X\to L_n^fX $. A spectrum $ X $ satisfies the $ E(n) $-telescope conjecture \cite[10.5]{\RavenelLocalization} if and only if the natural map $ L_n^fX\to L_n X $ is an equivalence (eg. \cite{\MahowaldRezk}).

Let $ X $ be a spectrum whose $ \F_p $-cohomology is finitely presented over the Steenrod algebra; the Mahowald-Rezk dual $ W_n $ is defined to be the Brown-Comenetz dual $ \IQZ C_n^f $. 

Suppose now that $X$ is an $ E(n) $-local spectrum which satisfies the $ E(n) $ and $ E(n-1) $-telescope conjectures. This condition is satisfied by the spectra of topological modular forms \cite[2.4.7]{\Behrens} with which we are concerned in this work. Then the monochromatic layer $ M_nX $ is the fiber of the natural map $ L_n^fX\to L_{n-1}^fX $. Taking the Brown-Comenetz dual of the first column in the diagram of fiber sequences
\[
\xymatrix{
\Sigma^{-1}M_nX \ar[r]\ar[d] &\star \ar[r]\ar[d] & M_n X\ar[d]\\
C_n^fX\ar[r]\ar[d] & X \ar[r]\ar@{=}[d] & L_n^f X\ar[d]\\
C_{n-1}^fX\ar[r] & X\ar[r] &L^f_{n-1}X
}
\]
implies the fiber sequence \cite[2.4.4]{\Behrens}
\begin{align*}
W_{n-1}X\to W_n X\to L_{K(n)}\IZ X,
\end{align*}
relating Mahowald-Rezk duality to $ K(n) $-local Anderson duality.

We have used the following result.
\begin{proposition}\label{prop:local}
For any $ X $ and $ Y $, the natural map $ F(L_nX,Y)\to F(M_nX,Y) $ is $ K(n) $-localization.
\end{proposition}
\begin{proof}
First of all, we need to show that $ F(M_nX,Y) $ is $ K(n) $-local. This is equivalent to the condition that for any $ K(n) $-acyclic spectrum $ Z $, the function spectrum $ F(Z,F(M_nX,Y))=F(Z\wedge M_nX, Y) $ is contractible. But the functor $ M_n $ is smashing, and is also the fiber of $ L_{K(n)}\to L_{ n-1}L_{K(n)}$, by the homotopy pull-back square \eqref{fracturesquare}.
Therefore, for $ K(n) $-acyclic $ Z $, $ M_nZ $ is contractible, and the claim follows.

It remains to show that the fiber $ F(L_{n-1}X,Y) $ is $ K(n) $-acyclic. To do this, smash with a generalized Moore spectrum $ T_n $. This is a finite, Spanier-Whitehead self-dual (up to a suspension shift) spectrum of type $ n $, which is therefore $ E(n-1) $-acyclic. A construction of such a spectrum can be found in Section 4 of \cite{\HoveyStrickland}. We have (up to a suitable suspension)
\begin{align*} F(L_{n-1}X,Y)\wedge T_n&=F(L_{n-1}X,Y)\wedge DT_n=F\bigl(T_n,F(L_{n-1}X,Y)\bigr)\\
&=F\bigl((L_{n-1}T_n)\wedge X,Y\bigr)=*,
\end{align*}
implying that $ F(L_{n-1}X,Y) $ is $ K(n) $-acyclic, which proves the proposition.
\end{proof}

\section{Derived Stacks}
In this section we briefly recall the notion of derived stack which
will be useful to us; a good general reference for the classical theory of stacks is \cite{\Champs}.
Roughly speaking, stacks arise as solutions to moduli problems by
allowing points to have nontrivial automorphisms. The classical
viewpoint is that a stack is a category fibered in groupoids. One can
also equivalently view a stack as a presheaf of groupoids
\cite{\Hollander}.

Deligne-Mumford stacks are particularly well-behaved, because they can be studied by maps from schemes in the following sense. If $ \Sh{X} $ is a Deligne-Mumford stack, then $ \Sh{X} $ has an \'etale cover by a scheme, and any map from a scheme $ S $ to $ \Sh{X} $ is representable. (A map of stacks $ f:\Sh{X}\to \Sh{Y} $ is said to be representable if for any scheme $ S $ over $ \Sh{Y} $, the ($2$-categorical) pullback $ \Sh{X}\cross{\Sh{Y}}S $ is again a scheme.) 

Derived stacks are obtained by allowing sheaves valued in a category
which has homotopy theory, for example differential graded algebras,
simplicial rings, or commutative ring spectra. To be able to make
sense of the latter, one needs a nice model for the category of
spectra and its smash product, with a good notion of commutative rings
and modules over those. Much has been written recently about derived
algebraic geometry, work of Lurie on the one hand
\cite{\Algebra,\Topos,\DAG} and To\"en-Vezossi on the other
\cite{\HAGI,\HAGII}. For this project we only consider sheaves of $
E_\infty $-rings on ordinary sites, and consequently we avoid the need to
work with infinity categories.  Derived Deligne-Mumford stacks will be
defined as follows.
\begin{definition}
A \emph{derived Deligne-Mumford stack} is a pair $
(\Sh{X},\SSh{\Sh{X}}^{top}) $ consisting of a topological space (or a Grothendieck topos) $
\Sh{X} $ and a sheaf $ \SSh{\Sh{X}}^{top} $ of $ E_\infty $-ring
spectra on its small \'etale site such that
\begin{enumerate}
\item[(a)] the pair $ (\Sh{X},\pi_0\SSh{\Sh{X}}^{top}) $ is an
ordinary Deligne-Mumford stack, and
\item[(b)] for every $ k $, $ \pi_k\SSh{\Sh{X}}^{top} $ is a
quasi-coherent $ \pi_0\SSh{\Sh{X}}^{top} $-module.
\end{enumerate}
\end{definition}

Here and elsewhere in this paper, by $ \pi_*\Sh{F} $ of a sheaf of $ E_\infty $-rings we will mean the sheafification of the presheaf $ U\mapsto \pi_*\Sh{F}(U) $.

Next we  discuss sheaves of $E_{\infty}$-ring spectra.
Let $ \Cat{C} $ be a small Grothendieck site. By a \emph{presheaf} of
$ E_\infty $-rings on $ \Cat{C} $ we shall mean simply a functor $
\Sh{F}:\Cat{C}^{op}\to \Rings $. The default example of a site
$\Cat{C} $ will be the small \'etale site $ \Sh{X}_{\acute{e}t} $ of a stack $ \Sh{X} $
\cite[Ch. 12]{\Champs}.

A presheaf $\Sh{F}$ of $ E_\infty $-rings on $\Cat{C}$ is said to satisfy {\em hyperdescent}
or that it is a {\em sheaf} provided that the natural map $
\Sh{F}(X)\to\holim \Sh{F}(U_{\bullet}) $ is a weak equivalence for
every hypercover $U_{\bullet}\to X$.
Hyperdescent is closely related to fibrancy in Jardine's model
category structure \cite{\JardineSimplicial,\JardineSpectra}.
Specifically if $F\to QF$ is a fibrant replacement in Jardine's model
structure, then \cite{\DHI} shows that $F$ satisfies hyperdescent if
and only if $F(U)\to QF(U)$ is a weak equivalence for all $U$.

When the site $\Cat{C}$ has enough points, one may use Godement
resolutions in order to ``sheafify" a presheaf  \cite[Section
3]{\JardineSimplicial}. In particular, since $ \Sh{X}_{\acute{e}t} $ has enough
points we may form the Godement resolution $\Sh{F}\to G\Sh{F}$.
The global sections of $ G\Sh{F} $ are called the derived global
sections of $ \Sh{F} $ and Jardine's construction also gives a
spectral sequence to compute the homotopy groups
\begin{align}\label{ss:jardine}
H^s(\Sh{X},\pi_t\Sh{F})\Rightarrow \pi_{t+s}R\Gamma\Sh{F}.
\end{align}

\section{Moduli of Elliptic Curves and Level Stuctures}\label{sec:moduli}
	
In this section, we summarize the results of interest regarding the moduli stacks of elliptic curves and level structures. Standard references for the ordinary geometry are Deligne-Rapoport \cite{\DeligneRapoport}, Katz-Mazur \cite{\KatzMazur}, and Silverman \cite{\Silverman}.

A \emph{curve} over a base scheme (or stack) $ S $ is a proper, flat morphism $ p: C\to S $, of finite presentation and relative dimension one. An \emph{elliptic curve} is a curve $ p:C\to S $ of genus one whose geometric fibers are connected, smooth, and proper, equipped with a section $ e:S\to C $ or equivalently, equipped with a commutative group structure \cite[2.1.1]{\KatzMazur} . These objects are classified by the moduli stack of elliptic curves $ \M^0 $.

The $ j $-invariant of an elliptic curve gives a map $ j:\M^0 \to \A^1$ which identifies $\A^1 $ with the coarse moduli scheme \cite[8.2]{\KatzMazur}. Thus $ \M^0 $ is not proper, and in order to have Grothendieck-Serre duality this is a property we need. Hence, we shall work with the compactification $ \M $ of $ \M^0 $, which has a modular description in terms of generalized elliptic curves: it classifies proper curves of genus one, whose geometric fibers are connected and allowed to have an isolated nodal singularity away from the point marked by the specified section $ e $ \cite[II.1.12]{\DeligneRapoport}.

If $ C $ is smooth, then the multiplication by $ n $ map $ [n]:C\to C $ is finite flat map of degree $ n^2 $ \cite[II.1.18]{\DeligneRapoport}, whose kernel we denote by $ C[n] $. If $ n $ is invertible in the ground scheme $ S $, the kernel $ C[n] $ is finite \'etale over $ S $, and \'etale locally isomorphic to $ (\Z/n)^2 $. A level $ n $ structure is then a choice of an isomorphism $ \phi:(\Z/n)^2\to C[n] $ \cite[IV.2.3]{\DeligneRapoport}. We denote the moduli stack (implicitly understood as an object over $ \Spec\Z[1/n] $) which classifies smooth elliptic curves equipped with a level $ n $ structure by $ \Ml{n}^0 $. 

In order to give a modular description of the compactification $ \Ml{n} $ of $ \Ml{n}^0 $, we need to talk about so-called N\'eron polygons. We recall the definitions from \cite[II.1]{\DeligneRapoport}. A \emph{(N\'eron) $ n $-gon} is a curve over an algebraically closed field isomorphic to the standard $ n $-gon, which is the quotient of $ \P^1\times \Z/n$ obtained by identifying the infinity section of the $ i $-th copy of $ \P^1 $ with the zero section of the $ (i+1) $-st. A curve $ p:C \to S $ is a \emph{stable curve of genus one} if its geometric fibers are either smooth, proper, and connected, or N\'eron polygons. A \emph{generalized elliptic curve} is a stable curve of genus one, equipped with a morphism $ +: C^{sm}\cross{S} C\to C$ which
\begin{enumerate}
\item[(a)] restricts to the smooth locus $ C^{sm} $ of $ C $, making it into a commutative group scheme, and
\item[(b)] gives an action of the group scheme $ C^{sm} $ on $ C $, which on the singular geometric fibers of $ C $ is given as a rotation of the irreducible components.
\end{enumerate}
Given a generalized elliptic curve $p:C\to S $, there is a locally finite family of disjoint closed subschemes $ \left(S_n\right)_{n\geq 1} $ of $ S $, such that $ C $ restricted to $ S_n $ is isomorphic to the standard $ n $-gon, and the union of all $ S_n $'s is the image of the singular locus $ C^{sing} $ of $ C $ \cite[II.1.15]{\DeligneRapoport}.

The morphism $ n:C^{sm}\to C^{sm} $ is again finite and flat, and if $ C $ is an $ m $-gon, the kernel $ C[n] $ is \'etale locally isomorphic to $ (\mu_n\times \Z/(n,m)) $ \cite[II.1.18]{\DeligneRapoport}. In particular, if $ C $ a generalized elliptic curve whose geometric fibers are either smooth or $ n $-gons, then the scheme of $ n $-torsion points $ C[n] $ is \'etale locally isomorphic to $ (\Z/n)^2 $. These curves give a modular interpretation of the compactification $ \Ml{n} $ of $ \Ml{n}^0 $. The moduli stack $ \Ml{n} $ over $ \Z[1/n] $ classifies generalized elliptic curves with geometric fibers that are either smooth or $ n $-gons, equipped with a level $ n $ structure, i.e. an isomorphism $ \varphi:(Z/n)^2\to C[n] $ \cite[IV.2.3]{\DeligneRapoport}.

Note that $ \Gl $ acts on $ \Ml{n} $ on the right by pre-composing with the level structure, that is, $g: (C,\varphi) \mapsto (C,\varphi\circ g)$. If $ C $ is smooth, this action is free and transitive on the set of level $ n $ structures. If $ C $ is an $ n $-gon, then the stabilizer of a given level $n $ structure is a subgroup of $ \Gl $ conjugate to the upper triangular matrices $U=\left\{\begin{pmatrix} 1& *\\0 & 1 \end{pmatrix}\right\}$.

The forgetful map $ \Ml{n}^0\to \M^0[1/n] $ extends to the compactifications, where it is given by forgetting the level structure and contracting the irreducible components that do not meet the identity section \cite[IV.1]{\DeligneRapoport}. The resulting map $ q:\Ml{n}\to \M[1/n] $ is a finite flat cover of $ \M[1/n] $ of degree $ |GL_2(\Z/n)| $. Moreover, the restriction of $ q $ to the locus of smooth curves is an \'etale $ \Gl $-torsor, and over the locus of singular curves, $ q $ is ramified of degree $ n $ \cite[A1.5]{\Katz}. In fact, at the ``cusps" the map $ q $ is given as \[ q:\Ml{n}^{\infty}\cong\Hom_{U}(\Gl,\M^{\infty})\to \M^{\infty}.\]

Note that level $ 1 $ structure on a generalized elliptic curve $ C/S $ is nothing but the specified identity section $ e: S\to C $. Thus we can think of $ \M $ as $ \Ml{1} $.

The objects $ \Ml{n} $, $ n\geq 1 $, come equipped with ($\Z$-)graded sheaves, the tensor powers $ \om{\Ml{n}}{*} $ of the sheaf of invariant differentials \cite[I.2]{\DeligneRapoport}. Given a generalized elliptic curve $ p:C\rightleftarrows S: e $, let $ \Sh{I} $ be the ideal sheaf of the closed embedding $ e $. The fact that $ C $ is nonsingular at $ e $ implies that the map $ p_*(\Sh{I}/\Sh{I}^2) \to p_*\omega_{C/S}$ is an isomorphism \cite[II.8.12]{\HartshorneAG}. Denote this sheaf on $ S $ by $ \omega_C $. It is locally free of rank one, because $ C $ is a curve over $ S $, with a potential singularity away from $ e $. The sheaf of invariant differentials on $ \Ml{n} $ is then defined by \[ \omega_{\Ml{n}}(S\xrightarrow{C}\Ml{n})=\omega_C. \] It is a quasi-coherent sheaf which is an invertible line bundle on $ \Ml{n} $.

The ring of modular forms with level $ n $ structures is defined to be the graded ring of global sections \[ MF(n)_*=H^0(\Ml{n},\om{\Ml{n}}{*}), \]
where, as usual, we denote $ MF(1)_* $ simply by $ MF_* $.

\section{Topological Modular Forms and Level Structures}\label{sec:dermoduli}

By the obstruction theory of Goerss-Hopkins-Miller, as well as work of Lurie, the moduli stack $\M$ has been upgraded to a derived Deligne-Mumford stack, in such a way that the underlying ordinary geometry is respected. Namely, a proof of the following theorem can be found in Mark Behrens' notes \cite{\ConstructTMF}.
\begin{thm}[Goerss-Hopkins-Miller, Lurie]\label{ConstructTMF}
The moduli stack $ \M $ admits a sheaf of $ E_\infty $-rings $ \Sh{O}^{top} $ on its \'etale site which makes it into a derived Deligne-Mumford stack. For an \'etale map $ \Spec R\to \M $ classifying a generalized elliptic curve $ C/R $, the sections $ \Sh{O}^{top}(\Spec R) $ form an even weakly periodic ring spectrum $ E $ such that
\begin{enumerate}
\item[(a)] $ \pi_0E = R $, and
\item[(b)] the formal group $ G_E $ associated to $ E $ is isomorphic to the completion $ \hat{C} $ at the identity section.
\end{enumerate}
Moreover, there are isomorphisms of quasi-coherent $\pi_0 \Sh{O}^{top} $-modules $ \pi_{2k}\Sh{O}^{top} \cong \om{\M}{k}$ and $ \pi_{2k+1}\Sh{O}^{top}\cong 0 $ for all integers $ k $.
\end{thm}

The spectrum of topological modular forms $ Tmf $ is defined to be the $ E_\infty $-ring spectrum of global sections $ R\Gamma(\Sh{O}^{top}) $.

We remark that inverting $ 6 $ kills all torsion in the cohomology of $ \M $ as well as the homotopy of $ Tmf $. In that case, following the approach of this paper would fairly formally imply Anderson self-duality for $ Tmf[1/6] $ from the fact that $ \M $ has Grothendieck-Serre duality. 
To understand integral duality on the derived moduli stack of generalized elliptic curves, we would like to use the strategy of descent, dealing separately with the $ 2 $ and $ 3 $-torsion. The case when $ 2 $ is invertible captures the $ 3 $-torsion, is more tractable, and is thoroughly worked out in this paper by using the smallest good cover by level structures, $ \Ml{2} $. 
The $ 2 $-torsion phenomena involve computations that are more daunting and will be dealt with subsequently. However, the same methodology works to give the required self-duality result.

To begin, we need to lift $ \Ml{2} $ and the covering map $ q:\Ml{2}\to\M $ to the setting of derived Deligne-Mumford stacks. We point out that this is not immediate from Theorem \ref{ConstructTMF} because the map $ q $ is not \'etale. However, we will explain how one can amend the construction to obtain a sheaf of $ E_\infty $-rings $ \Sh{O}(2)^{top} $ on $ \Ml{2} $. We will also incorporate the $ GL_2(\Z/2) $-action, which is crucial for our result.

We will in fact sketch an argument based on Mark Behrens' \cite{\ConstructTMF} to construct sheaves $ \Sh{O}(n)^{top} $ on $ \Ml{n}[1/2n] $ for any $ n $, as the extra generality does not complicate the solution.

As we remarked earlier, the restriction of $q $ to the smooth locus $\Ml{n}^{0}$ is an \'etale $\Gl$-torsor, hence we automatically obtain $\Sh{O}(n)|_{\Ml{n}^{0}}$ together with its $\Gl$-action. We will use the Tate curve and $ K(1) $-local obstruction theory to construct the $ E_\infty $-ring of sections of the putative $ \Sh{O}(n)^{top} $ in a neighborhood of the cusps, and sketch a proof of the following theorem.

\begin{thm}[Goerss-Hopkins]
The moduli stack $ \Ml{n} $ (as an object over $ \Z[1/2n] $) admits a sheaf of even weakly periodic $ E_\infty $-rings $ \Sh{O}(n)^{top} $ on its \'etale site which makes it into a derived Deligne-Mumford stack. There are isomorphisms of quasi-coherent $ \pi_0\Sh{O}(n)^{top} $-modules $ \pi_{2k}\Sh{O}(n)^{top}\cong \om{\Ml{n}}{k} $ and $ \pi_{2k-1}\Sh{O}(n)^{top}\cong 0 $ for all integers $ k $.
Moreover, the covering map $ q:\Ml{n}\to \M[1/2n] $ is a map of derived Deligne-Mumford stacks.
\end{thm}

\subsection{Equivariant $ K(1) $-local Obstruction Theory}

This is a combination of the Goerss-Hopkins' \cite{\GoerssHopkins, \GH} and Cooke's \cite{\Cooke} obstruction theories, which in fact is contained although not explicitly stated in \cite{\GoerssHopkins,\GH}.
Let $ G $ be a finite group; a $ G $-equivariant $ \theta $-algebra is an algebraic model for the $ p $-adic $ K $-theory of an $ E_\infty $ ring spectrum with an action of $ G $ by $ E_\infty $-ring maps. Here, $ G $-equivariance means that the action of $ G $ commutes with the $ \theta $-algebra operations. As $ G $-objects are $ G $-diagrams, the obstruction theory framework of \cite{\GH} applies.

Let $ H^s_{G-\theta}(A,B[t]) $ denote the $ s $-th derived functor of derivations from $ A $ into $ B[t] $\footnote{For a graded module $ B $, $ B[t] $ is the shifted graded module with $ B[t]_k=B_{t+k}$.} in the category of $ G $-equivariant $ \theta $-algebras. These are the $ G $-equivariant derivatons from $ A $ into $ B[t] $, and there is a composite functor spectral sequence
\begin{align}\label{GObstr}
H^r(G, H^{s-r}_\theta(A,B[t]))\Rightarrow H^s_{G-\theta}(A,B[t]).
\end{align}

\begin{thm}[Goerss-Hopkins]\label{Obstruct}
\begin{enumerate}
\item[\textit{(a)}] Given a $ G-\theta $-algebra $ A $, the obstructions to existence of a $ K(1) $-local even-periodic $ E_\infty $-ring $ X $ with a $ G $-action by $ E_\infty $-ring maps, such that $ K_*X\cong A $ (as $ G-\theta $-algebras) lie in \[ H^s_{G-\theta}(A,A[-s+2]) \text{ for } s\geq 3.\]The obstructions to uniqueness like in \[ H^s_{G-\theta}(A,A[-s+1]) \text{ for } s\geq 2. \]
\item[\textit{(b)}] Given $ K(1) $-local $ E_\infty $-ring $ G $-spectra $ X$ and $ Y $ whose $ K $-theory is $ p $-complete, and an equivariant map of $ \theta $-algebras $ f_*:K_*X\to K_* Y $, the obstructions to lifting $ f_* $ to an equivariant map of $ E_\infty $-ring spectra lie in \[ H^s_{G-\theta}(K_*X,K_*Y[-s+1]), \text{ for } s\geq 2, \]
while the obstructions to uniqueness of such a map lie in 
\[ H^s_{G-\theta}(K_*X,K_*Y[-s]), \text{ for } s\geq 1. \]
\item[\textit{(c)}] Given such a map $ f:X\to Y $, there exists a spectral sequence \[ H^s_{G-\theta}(K_*X,K_*Y[t])\Rightarrow \pi_{-t-s}(E_\infty(X,Y)^G,f), \]
computing the homotopy groups of the space of equivariant $ E_\infty $-ring maps, based at $ f $.
\end{enumerate}
\end{thm}

\subsection{The Igusa Tower}

Fix a prime $ p>2 $ which does not divide $ n $. Let $ \M_p $ denote the $ p $-completion of $ \M $, and let $\Msing$ denote a formal neighborhood of the cusps of $\M_{p}$. We will use the same embellishments for the moduli with level structures.

The idea is to use the above Goerss-Hopkins obstruction theory to construct the $ E_\infty $-ring of sections over $ \Ml{n}^\infty $ of the desired $ \Sh{O}(n)^{top} $, from the algebraic data provided by the Igusa tower, which will supply an equivariant $ \theta $-algebra.

We will consider the moduli stack $ \Ml{n} $ as an object over $ \Z[1/n,\zeta] $, for $ \zeta $ a primitive $ n $-th root of unity. The structure map $ \Ml{n}\to \Spec \Z[1/n,\zeta]  $ is given by the Weil pairing \cite[5.6]{\KatzMazur}.

The structure of $ \M^\infty $ as well as $ \Ml{n}^\infty $ is best understood by the Tate curves. We already mentioned that the singular locus is given by N\'eron $ n $-gons; the $ n $-Tate curve is a generalized elliptic curve in a neighborhood of the singular locus. It is defined over the ring $ \Z[[q^{1/n}]] $, so that it is smooth when $ q $ is invertible and a N\'eron $ n $-gon when $ q $ is zero. For details of the construction, the reader is referred to \cite[VII]{\DeligneRapoport}.

>From \cite[VII]{\DeligneRapoport} and \cite[Ch 10]{\KatzMazur}, we learn that $ \Msing=\Spf \Z_p[[q]] $, and that $ \displaystyle{\Mlsing{n}=\coprod_{cusps}\Spf\Z_p[\zeta][[q^{1/n}]]} $. The $ \Gl $-action on $ \Mlsing{n} $ is understood by studying level structures on Tate curves, and is fully described in \cite[Theorem 10.9.1]{\KatzMazur}. The group $ U $ of upper-triangular matrices in $ \Gl $ acts on $B(n)= \Z_p[\zeta][[q^{1/n}]] $ by the roots of unity: $\left(\begin{array}{cc} 1 & a \\ 0 & 1 \end{array}\right)\in U$ sends $ q^{1/n} $ to $ \zeta^a q^{1/n} $. Note that then, the inclusion $ \Z_p[[q]]\to B(n) $ is $ U $-equivariant, and $ \Mlsing{n} $ is represented by the induced $ \Gl $-module $ A(n)=\Hom_U(\Gl,B(n)) $.

Denote by $ \Mlsing{p^k,n} $ a formal neighborhood of the cusps in the moduli stack of generalized elliptic curves equipped with level structure maps
\begin{align*}
\eta:\mu_{p^k}\xrightarrow{\sim} C[p^k]\\
\varphi:(\Z/n)^2\xrightarrow{\sim} C[n].
\end{align*}
Note that as we are working in a formal neighborhood of the singular locus, the curves $ C $ classified by $ \Ml{p^k,n}^\infty $ have ordinary reduction modulo $ p $. 
As $ k $ runs through the positive integers, we obtain an inverse system called the Igusa tower, and we will write $ \Ml{p^\infty,n} $ for the inverse limit. It is the $ \Z_p^\times $-torsor over $ \Mlsing{n} $ given by the formal affine scheme $ \Spf\Hom(\Z_p^\times,A(n)) $ \cite[Theorem 12.7.1]{\KatzMazur}.

The Tate curve comes equipped with a canonical invariant differential, which makes $ \om{\Mlsing{n}}{} $ isomorphic to the line bundle associated to the graded $ A(n) $-module $ A(n)[1] $. We define $ A(n)_* $ to be the evenly graded $ A(n) $-module, which in degree $ 2t $ is $ H^0(\Mlsing{n},\om{\Mlsing{n}}{t})\cong A(n)[t] $. Similarly we define $ A(p^\infty,n)_* $.

The modules $A(p^\infty,n)_*= \Hom(\Z_p^\times,A(n)_*) $ have a natural $\Gl$-equivariant $\theta $-algebra structure,  with operations coming from the following maps
\begin{align*}
(\psi^k)^*:\Ml{p^\infty, n}&\to \Ml{p^\infty,n}\\
(C,\eta,\varphi)&\mapsto (C,\eta\circ[k],\varphi),\\
(\psi^p)^*:\Ml{p^\infty, n}&\to \Ml{p^\infty,n}\\
(C,\eta,\varphi)&\mapsto (C^{(p)},\eta^{(p)},\varphi^{(p)}).
\end{align*}
Here, $ C^{(p)} $ is the elliptic curve obtained from $ C $ by pulling back along the absolute Frobenius map on the base scheme. The multiplication by $ p $ isogeny $ [p]:C\to C $ factors through $ C^{(p)} $ as the relative Frobenius $ F:C\to C^{(p)} $ followed by the Verschiebung map $ V:C^{(p)}\to C $ \cite[12.1]{\KatzMazur}. The level structure $ \varphi^{(p)}:(\Z/n)^2\to C^{(p)}$ is the unique map making the diagram
\[ \xymatrix{
& (\Z/n)^2 \ar[d]^{\varphi}\ar@{-->}[ld]_{\varphi^{(p)}}\\
 C^{(p)} \ar[r]^-V &C
}\]
commute. Likewise, $ \eta^{(p)} $ is the unique level structure making an analogous diagram commute. More details on $ \eta^{(p)} $ are in \cite[Section 5]{\ConstructTMF}. The operation $ \psi^k $ comes from the action of $ \Z_p^\times $ on the induced module $ A(p^\infty,n)_*= \Hom(\Z_p^\times,A(n)_*)  $.

We point out that from this description of the operations, it is clear that the $ \Gl $-action commutes with the operations $ \psi^k $ and $ \psi^p $, which in particular gives an isomorphism
\begin{align}\label{EqInd}
 A(p^\infty, n)_*  \cong \Hom_U\left(G,\Hom\left(\Z_p^\times,B(n)_*\right) \right). 
 \end{align}
Denote by $ B(p^\infty,n)_* $ the $ \theta $-algebra $ \Hom(\Z_p^\times, B(n)_* ) $.

As a first step, we apply (a) of Theorem \ref{Obstruct} to construct even-periodic $ K(1) $-local $ \Gl $-equivariant $ E_\infty $-ring spectra $ Tmf(n)^\infty_p $ whose $ p $-adic $ K $-theory is given by $ A(p^\infty, n)_* $. The starting point is the input to the spectral sequence \eqref{GObstr}, the group cohomology
\begin{align}\label{ss:GObstruct}
H^r\left(\Gl, H^{s-r}_\theta\left(A(p^\infty,n)_*,A(p^\infty,n)_*\right)\right ).
\end{align}
\begin{rem}\label{rm:TateKthry}
Lemma 7.5 of \cite{\ConstructTMF } implies that $ H^s_\theta(B(p^\infty,n)_*,B(p^\infty,n)_*)=0 $ for $ s>1 $, from which we deduce the existence of a unique $ K(1) $-local weakly even periodic $ E_\infty $-ring spectrum that we will denote by $ K[[q^{1/n}]] $, whose $ K $-theory is the $ \theta $-algebra $ B(p^\infty,n)_* $. This spectrum $ K[[q^{1/n}]] $ should be thought of as the sections of $ \Sh{O}(n)^{top} $ over a formal neighborhood of a single cusp of $ \Ml{n} $.
\end{rem}

By the same token, we also know that $ H^s_\theta(A(p^\infty,n)_*,A(p^\infty,n)_*)=0 $ for $ s>1 $, and \[ H^0_{\theta}\left(A(p^\infty,n)_*,A(p^\infty,n)_*\right)\cong\Hom_U\left(\Gl,H^0_\theta\left(A(p^\infty,n)_*,B(p^\infty,n)_* \right)\right). \]
Thus the group cohomology \eqref{ss:GObstruct} is simply
\[ H^r\left(U, H^0_\theta\left(A(p^\infty,n)_*,B(p^\infty,n)_*\right)\right) \]
which is trivial for $ r>0 $ as the coefficients are $ p $-complete, and the group $ U $ has order $ n $, coprime to $ p $.
Therefore, the spectral sequence \eqref{GObstr} collapses to give that 
\[ H^s_{\Gl-\theta}\left(A(p^\infty,n)_*,A(p^\infty,n)_* \right) = 0, \text{ for } s>0.\]
Applying Theorem \ref{Obstruct} now gives our required $ \Gl $-spectra $ Tmf(n)^\infty_p $.

A similar argument produces a $ \Gl $-equivariant $ E_\infty $-ring map \[ q^\infty: Tmf^\infty_p\to Tmf(n)^\infty_p,\]where $ Tmf^\infty $ is given the trivial $ \Gl $-action.

\begin{proposition}\label{prop:htpyfixed}
The map $ q^\infty: Tmf^\infty_p\to Tmf(n)^\infty_p $ is the inclusion of homotopy fixed points.
\end{proposition}
\begin{proof}
Note that from our construction it follows that $ Tmf(n)^\infty_p $ is equivalent to $ \Hom_U(\Gl,K[[q^{1/n}]]) $, where $ K[[q^{1/n}]] $ has the $ U $-action lifting the one we described above on its $ \theta $-algebra $ B(n)_* $. (This action lifts by obstruction theory to the $ E_\infty $-level because the order of $ U $ is coprime to $ p $.) Since $ Tmf^\infty_p $ has trivial $ \Gl $-action, the map $ q^\infty $ factors through the homotopy fixed point spectrum
\begin{align*}
\left( Tmf(n)^\infty_p \right)^{h\Gl} \cong K[[q^{1/n}]]^{hU}. 
\end{align*}
So we need that the map $ q': Tmf^\infty = K[[q]] \to K[[q^{1/n}]]^{hU}$ be an equivalence. The homotopy groups of $ K[[q^{1/n}]]^{hU} $ are simply the $ U $-invariant homotopy in $ K[[q^{1/n}]] $, because $ U $ has no higher cohomology with $ p $-adic coefficients. Thus $\pi_* q' $ is an isomorphism, and the result follows.
\end{proof}

\subsection{Gluing}

We need to patch these results together to obtain the sheaf $ {\Sh{O}}(n)^{top} $ on the \'etale site of $ \Ml{n} $.

To construct the presheaves $ \tilde{\Sh{O}}(n)^{top}_p $ on the site of affine schemes \'etale over $ \Ml{n} $, one follows the procedure of \cite[Step 2, Sections 7 and 8]{\ConstructTMF}. Thus for each prime $ p>2 $ which does not divide $ n $, we have $ \tilde{\Sh{O}}(n)_p $ on $ \Ml{n}_p $. As in \cite[Section 9]{\ConstructTMF}, rational homotopy theory produces a presheaf $ \tilde{\Sh{O}}(n)_{\Q} $. These glue together to give $ \tilde{\Sh{O}}(n) $ (note, $ 2n $ will be invertible in $ \tilde{\Sh{O}}(n) $).

By construction, the homotopy group sheaves of $ \tilde{\Sh{O}}(n)^{top} $ are given by the tensor powers of the sheaf of invariant differentials $ \om{\Ml{n}}{} $, which is a quasi-coherent sheaf on $ \Ml{n} $. Section $2$ of \cite{\ConstructTMF} explains how this data gives rise to a sheaf $ \Sh{O}(n)^{top} $ on the \'etale site of $ \Ml{n} $, so that the $ E_\infty $-ring spectrum $ Tmf(n) $ of global sections of $ \Sh{O}(n)^{top} $ can also be described as follows.

Denote by $ TMF(n) $ the $ E_\infty $-ring spectrum of sections $R\Gamma\left(\Sh{O}(n)^{top}|_{\Ml{n}^0}\right) $ over the locus of smooth curves. Let $ TMF(n)^\infty $ be the the $ E_\infty $-ring of sections of $ \Sh{O}(n)^{top}|_{\Ml{n}^0} $ in a formal neighborhood of the cusps. Both have a $  \Gl $-action by construction.
The equivariant obstruction theory of Theorem \ref{GObstr} can be used again to construct a $ \Gl $-equivariant map $ TMF(n)_{K(1)} \to TMF(n)^\infty_p$ that refines the $ q $-expansion map; again, the key point as that $ K_* (TMF(n)^\infty_p )$ is a $ U $-induced $ \Gl $-module. Pre-composing with the $ K(1)$-localization map, we get a $ \Gl $-$ E_\infty $-ring map $ TMF(n)\to TMF(n)^\infty $.

We build $ Tmf(n)_p $ as a pullback
\begin{align*}
\xymatrix{
Tmf(n)_p \ar[r]\ar[d] & Tmf(n)^\infty_p\ar[d] \\
TMF(n)_p \ar[r] & TMF(n)^{\infty}_p}
\end{align*}
All maps involved are $ \Gl $-equivariant maps of $ E_\infty $-rings, so $ Tmf(n) $ constructed this way has a $\Gl $-action as well.

\section{Descent and Homotopy Fixed Points}

We have remarked several times that the map $ q^0:\Ml{n}^0\to \M^0 $ is a $ \Gl $-torsor, thus we have a particularly nice form of \'etale descent. On global sections, this statement translates to the equivalence
\[ TMF[1/2n] \to TMF(n)^{h\Gl}.\]
The remarkable fact is that this property goes through for the compactified version as well.

\begin{thm}\label{thm:tmffixedpt}
The map $ Tmf[1/2n]\to Tmf(n)^{h\Gl} $ is an equivalence.
\end{thm}
\begin{proof}
This is true away from the cusps, but by Proposition \ref{prop:htpyfixed}, it is also true near the cusps. We constructed $ Tmf(n) $ from these two via pullback diagrams, and homotopy fixed points commute with pullbacks.
\end{proof}

\begin{rem}
For the rest of this paper, we will investigate $ Tmf(2) $, the spectrum of topological modular forms with level $ 2 $ structure. Note that this spectrum differs from the more commonly encountered $ TMF_0(2) $, which is the receptacle for the Ochanine genus \cite{\Ochanine}, as well as the spectrum appearing in the resolution of the $ K(2) $-local sphere \cite{\GHMR,\Behrens}. The latter is obtained by considering isogenies of degree $ 2 $ on elliptic curves, so-called $ \Gamma_0(2) $ structures.
\end{rem}

\section{Level $ 2 $ Structures Made Explicit}\label{sec:leveltwo}

In this section we find an explicit presentation of the moduli stack $ \Ml{2} $.

Let $ E/S $ be a generalized elliptic curve over a scheme on which $ 2 $ is invertible, and whose geometric fibers are either smooth or have a nodal singularity (i.e. are N\'eron $ 1 $-gons). Then Zariski locally, $ E $ is isomorphic to a Weierstrass curve of a specific and particularly simple form. Explicitly, there is a cover $ U\to S $ and functions $ x,y $ on $ U $ such that the map $U\to\P^2_U $ given by $ [x,y,1] $ is an isomorphism between $ E_U=E\cross{S}U $ and a Weierstrass curve in $ \P^2_U $ given by the equation
\begin{align}\label{Eb}
E_b:\quad y^2=x^3+\frac{b_2}{4}x^2+\frac{b_4}{2}x+\frac{b_6}{4}=:f_b(x),
\end{align}
such that the identity  for the group structure on $ E_U $ is mapped to the point at infinity $ [0,1,0] $\cite[III.3]{\Silverman},\cite{\KatzMazur}. Any two Weierstrass equations for $ E_U $ are related by a linear change of variables of the form
\begin{equation}\label{transf}
\begin{split}
x&\mapsto u^{-2}x+r\\
y&\mapsto u^{-3}y.
\end{split}
\end{equation}

The object which classifies locally Weierstrass curves of the form \eqref{Eb}, together with isomorphisms which are given as linear change of variables \eqref{transf}, is a stack $ \M_{weier}[1/2] $, and the above assignment $ E\mapsto E_b $ of a locally Weierstrass curve to an elliptic curve defines a map $ w: \M[1/2]\to\M_{weier}[1/2] $. 

The Weierstrass curve \eqref{Eb} associated to a generalized elliptic curve $ E $ over an algebraically closed field has the following properties: $ E $ is smooth if and only if the discriminant of $ f_b(x) $ has no repeated roots, and $ E $ has a nodal singularity if and only if $ f_b(x) $ has a repeated but not a triple root. Moreover, non-isomorphic elliptic curves cannot have isomorphic Weierstrass presentations. Thus the map $ w: \M[1/2]\to\M_{weier}[1/2] $ injects $ \M[1/2] $ into the open substack $ U(\Delta) $ of $ \M_{weier}[1/2] $ which is the locus where the discriminant of $ f_b $ has order of vanishing at most one.

Conversely, any Weierstrass curve of the form \eqref{Eb} has genus one, is smooth if and only if $ f_b(x) $ has no repeated roots, and has a nodal singularity whenever it has a double root, so $ w:\M[1/2]\to U(\Delta) $ is also surjective, hence an isomorphism. Using this and the fact that points of order two on an elliptic curve are well understood, we will find a fairly simple presentation of $ \Ml{2} $.

The moduli stack of locally Weierstrass curves is represented by the Hopf algebroid \[ (B=\Z[1/2][b_2,b_4,b_6],B[u^{\pm 1},r]).\] Explicitly, there is a presentation $ \Spec B\to \M_{weier}[1/2]$, such that \[ \Spec B \cross{\M_{weier}[1/2]}\Spec B=\Spec B[u^{\pm 1},r]. \]
The projection maps to $ \Spec B $ are $ \Spec $ of the inclusion of $ B $ in $ B[u^{\pm 1}] $ and $ \Spec $ of the map
\begin{align*}
b_2&\mapsto u^2(b_2+12r)\\
b_4&\mapsto u^4(b_4+rb_2+6r^2)\\
b_6&\mapsto u^6(b_6+2rb_4+r^2b_2+4r^3)
\end{align*}
which is obtained by plugging in the transformation \eqref{transf} into \eqref{Eb}. In other words, $ \M_{weier}[1/2] $ is simply obtained from $ \Spec B $ by enforcing the isomorphisms that come from the change of variables \eqref{transf}.

Suppose $ E/S $ is a smooth elliptic curve which is given locally as a Weierstrass curve \eqref{Eb}, and let $\phi:(\Z/2)^2\to E$ be a level $ 2 $ structure. For convenience in the notation, define $e_0=\binom{1}{1},e_1=\binom{1}{0},e_2=\binom{0}{1}\in(\Z/2)^2$. Then $\phi(e_i)$ are all points of exact order $2$ on $E$, thus have $ y $-coordinate equal to zero since $ [-1](y,x)=(-y,x) $ (\cite[III.2]{\Silverman}) and \eqref{Eb} becomes
\begin{align}\label{Ex}
y^2=(x-x_0)(x-x_1)(x-x_2),
\end{align}
where $x_i=x(\phi(e_i))$ are all different.

If $ E $ is a generalized elliptic curve which is singular, i.e. $ E $ is a N\'eron $ 2 $-gon, then a choice of level $ 2 $ structure makes $ E $ locally isomorphic to the blow-up of \eqref{Ex} at the singularity (seen as a point of $ \P^2 $), with $ x_i=x_j\neq x_k $, for $ \{i,j,k\}=\{0,1,2\} $.

So let $A=\Z[1/2][x_0,x_1,x_2]$, let $ L $ be the line in $ \Spec A $ defined by  the ideal $(x_0-x_1,x_1-x_2,x_2-x_0)$, and let $\Spec A-L$ be the open complement. The change of variables \eqref{transf} translates to a $ (\G_a\rtimes\G_m) $-action on $ \Spec A $ that preserves $ L $ and is given by:
\begin{align*}
x_i&\mapsto u^2(x_i-r).
\end{align*}
Consider the isomorphism $\psi:(\Spec A-L)\xrightarrow{\sim}(\A^2-\{0\})\times\A^1$:
\[
(x_0,x_1,x_2)\mapsto\left( (x_1-x_0,x_2-x_0),x_0\right).
\]
We see that $\G_a$ acts trivially on the $(\A^2-\{0\})$-factor, and freely by translation on $\A^1$. Therefore the quotient $(\Spec A-L)//\G_a$ is \[\Mc2=\A^2-\{0\}=\Spec\Z[1/2][\lambda_1,\lambda_2]-\{0\},\]the quotient map being $\psi$ composed with the projection onto the first factor. This corresponds to choosing coordinates in which $E$ is of the form:
\begin{equation}\label{Elambda}
y^2=x(x-\lambda_1)(x-\lambda_2).
\end{equation}
The $ \G_m $-action is given by grading $ A $ as well as $ \Lambda=\Z[1/2][\lambda_1,\lambda_2] $ so that the degree of each $ x_i $ and $ \lambda_i $ is $ 2 $. It follows that $\Ml{2}=\Mc2//\G_m$ is the weighted projective line $\Proj\Lambda=(\Spec\Lambda-\{0\})//\G_m$. Note that we are taking homotopy quotient which makes a difference: $ -1 $ is a non-trivial automorphism on $ \Ml{2} $ of order $ 2 $.

The sheaf $ \omega_{\Ml{2}} $ is an ample invertible line bundle on $ \Ml{2} $, locally generated by the invariant differential $\displaystyle{\eta_{E_\lambda}= \frac{dx}{2y} }$. From \eqref{transf} we see that the $ \G_m=\Spec\Z[u^{\pm 1}] $-action changes $ \eta_{E_\lambda}$ to $ u\eta_{E_\lambda} $. Hence, $ \omega_{\Ml{2}} $ is the line bundle on $ \Ml{2}=\Proj\Lambda $ which corresponds to the shifted module $ \Lambda[1] $, standardly denoted by $ \Sh{O}(1) $. 
We summarize the result.
\begin{proposition}\label{M2stack}
The moduli stack of generalized elliptic curves with a choice of a level $ 2 $ structure $ \Ml{2}$ is isomorphic to $\Proj\Lambda=(\Spec\Lambda-\{0\})//\G_m $, via the map $ \Ml{2} \to\Proj\Lambda$ which classifies the sheaf of invariant differentials $ \omega_{\Ml{2}} $on $ \Ml{2} $. 
The universal curve over the locus of smooth curves $ \Ml{2}^0 =\Proj\Lambda-\{0,1,\infty\}$ is the curve of equation \eqref{Elambda}. The fibers at $ 0$, $1$, and $\infty $, are N\'eron $ 2 $-gons obtained by blowing up the singularity of the curve \eqref{Elambda}.
\end{proposition}

\begin{rem}\label{M2gr}
As specifying a $ \G_m $-action is the same as specifying a grading, we can think of the ringed space $ (\Ml{2},\om{\Ml{2}}{*}) $ as the ringed space $ (\Mc2=\Spec\Lambda-\{0\},\SSh{\Mc2}) $ together with the induced grading.
\end{rem}

Next we proceed to understand the action of $ GL_2(\Z/2) $ on the global sections $ H^0(\Ml{2},\om{\Ml{2}}{*})=\Lambda $. By definition, the action comes from the natural action of $ GL_2(\Z/2) $ on $ (\Z/2)^2 $ and hence by pre-composition on the level structure maps $ \phi:(\Z/2)^2\to E[2] $. If we think of $ GL_2(\Z/2) $ as the symmetric group $ S_3 $, then this action permutes the non-zero elements $ \{e_0,e_1,e_2\} $ of $ (\Z/2)^2 $, which translates to the action on
\[
H^0(\Spec A-L,\SSh{\Spec A-L})=\Z[x_0,x_1,x_2],
\]
given as  $g\cdot x_i=x_{gi}$ where $g\in S_3=\Perm\{0,1,2\}$. The map on $H^0$ induced by the projection $(\Spec A-L)\to \Mc2$ is
\begin{align*}
\Z[\lambda_1,\lambda_2]&\to\Z[x_0,x_1,x_2]\\
\lambda_i&\mapsto x_i-x_0
\end{align*}
Therefore, we obtain that $g\lambda_i$ is the inverse image of $x_{gi}-x_{g0}$. That is, $g\lambda_i=\lambda_{gi}-\lambda_{g0}$, where we implicitly understand that $\lambda_0=0$. We have proven the following lemma.

\begin{lemma}\label{ActionLambda}
Choose the generators of $S_3=\Perm\{0,1,2\}$, $\sigma=(012)$ and $\tau=(12)$. Then the $ S_3 $ action on $ \Lambda=H^0(\Ml{2},\om{\Ml{2}}{*}) $ is determined by
\begin{equation*}
\begin{aligned}
\tau:\quad&\lambda_1\mapsto\lambda_2 &\sigma:\quad&\lambda_1\mapsto\lambda_2-\lambda_1\\
&\lambda_2\mapsto\lambda_1 & &\lambda_2\mapsto-\lambda_1.
\end{aligned}
\end{equation*}
\end{lemma}
This fully describes the global sections $H^0(\Ml{2},\om{\Ml{2}}{*})$ as an $S_3$-module. The action on $H^1(\Ml{2},\om{\Ml{2}}{*})$ is not as apparent and we deal with it using Serre duality \eqref{GSDuality}.

\section{(Equivariant) Serre Duality for $ \Ml{2} $}

We will proceed to prove Serre duality for $ \Ml{2} $ in an explicit manner that will be useful later, by following the standard computations for projective spaces, as in \cite{\HartshorneAG}. To emphasize the analogy with the corresponding statements about the usual projective line, we might write $ \Proj\Lambda $ and $ \Sh{O}(*) $ instead of $ \Ml{2} $ and $ \om{\Ml{2}}{*} $, in view of Remark \eqref{M2gr}. Also remember that for brevity, we might omit writing $ 1/2 $.

\begin{proposition}\label{M2cohomology}
The cohomology of $ \Ml{2} $ with coefficients in the graded sheaf of invariant differentials $ \om{\Ml{2}}{*} $ is computed as
\begin{align*}
H^s(\Ml{2},\om{\Ml{2}}{*})=\begin{cases}
\Lambda,&s=0\\
\Lambda/(\lambda_1^\infty,\lambda_2^\infty),&s=1\\
0,&\text{else}
\end{cases}
\end{align*}
where $ \Lambda/(\lambda_1^\infty,\lambda_2^\infty) $ is a torsion $ \Lambda $-module with a $ \Z[1/2] $-basis of monomials $ \frac{1}{\lambda_1^{i}\lambda_2^j} $ for $ i,j $ both positive. 
\end{proposition}
\begin{rem}
The module $ \Lambda/(\lambda_1^\infty,\lambda_2^\infty) $ is inductively defined by the short exact sequences
\begin{align*}
0\to & \Lambda \xrightarrow{\lambda_1} \Lambda\Bigl[\frac{1}{\lambda_1}\Bigr]\to \Lambda/(\lambda_1^\infty)\to 0\\
0\to &\Lambda/(\lambda_1^\infty) \xrightarrow{\lambda_2} \Lambda\Bigl[\frac{1}{\lambda_1\lambda_2}\Bigr]\to \Lambda/(\lambda_1^\infty,\lambda_2^\infty)\to 0.
\end{align*}
\end{rem}
\begin{rem}
Note that according to \ref{M2gr}, $ H^*(\Ml{2},\om{\Ml{2}}{*})$ is isomorphic to $ H ^ * ( \Mc2, \SSh { \Mc2 } ) $ with the induced grading. It is these latter cohomology groups that we shall compute.
\end{rem}

\begin{proof}
We proceed using the local cohomology long exact sequence \cite[Ch III, ex. 2.3]{\HartshorneAG} for \[ \Mc2\subset \Spec\Lambda \supset \{0\}.\]
The local cohomology groups $ R^*\Gamma_{\{0\}}(\Spec\Lambda,\Sh{O})$ are computed via a Koszul complex as follows. The ideal of definition for the point $ \{0\} \in\Spec\Lambda$ is $ (\lambda_1,\lambda_2) $, and the generators $ \lambda_i $ form a regular sequence. Hence, $ R^*\Gamma_{\{0\}}(\Spec\Lambda,\Sh{O}) $ is the cohomology of the Koszul complex
\begin{align*}
\Lambda \to \Lambda\Bigl[\frac{1}{\lambda_1}\Bigr]\times \Lambda\Bigl[\frac{1}{\lambda_2}\Bigr]\to \Lambda\Bigl[\frac{1}{\lambda_1\lambda_2}\Bigr],
\end{align*}
which is $ \Lambda/(\lambda_1^\infty,\lambda_2^\infty) $, concentrated in (cohomological) degree two. We also know that $ H^*(\Spec\Lambda,\Sh{O})=\Lambda $ concentrated in degree zero, so that the local cohomology long exact sequence splits into
\begin{align*}
0&\to\Lambda\to H^0(\Mc2,\Sh{O})\to 0\\
0&\to H^1(\Mc2,\Sh{O})\to \Lambda/(\lambda_1^\infty,\lambda_2^\infty)\to 0,
\end{align*}
giving the result.
\end{proof}

\begin{lemma}\label{M2Omega}
We have the following properties of the sheaf of differentials on $ \Ml{2} $.
\begin{enumerate}
\item[(a)] There is an isomorphism $ \Om{\Ml{2}}\cong\om{\Ml{2}}{-4} $.
\item[(b)] The cohomology group $ H^s(\Ml{2},\Om{\Ml{2}}) $ is zero unless $ s=1 $, and\\ $ H^1(\Ml{2},\Om{\Ml{2}})$ is the sign representation $ \Z_{\sgn}[1/2] $ of $ S_3 $.
\end{enumerate}
\end{lemma}
\begin{proof} 
\begin{enumerate}
\item[(a)] The differential form $\eta= \lambda_1d\lambda_2-\lambda_2d\lambda_1 $ is a nowhere vanishing differential form of degree four, thus a trivializing global section of the sheaf $\Sh{O}(4)\otimes\Om{\Proj\Lambda} $. Hence there is an isomorphism $ \Om{\Proj\Lambda} \cong\Sh{O}(-4)$.
\item[(b)]  From Proposition \ref{M2cohomology}, $ H^*(\Proj\Lambda,\Omega)$ is $\Z[1/2] $ concentrated in cohomological degree one, and generated by \[ \frac{\eta}{\lambda_1\lambda_2}=\frac{\lambda_1}{\lambda_2}d\Bigl(\frac{\lambda_2}{\lambda_1}\Bigr). \]
Any projective transformation $ \varphi $ of $ \Proj\Lambda $ acts on $ H^1(\Proj\Lambda,\Omega) $ by the determinant $ \det\varphi $. By our previous computations, as summarized in Lemma \ref{ActionLambda}, the transpositions of $ S_3 $ act with determinant $ -1 $, and the elements of order $ 3 $ of $ S_3 $ with determinant $ 1 $. Hence the claim.
\end{enumerate}
\end{proof}

We are now ready to state and prove the following result.

\begin{thm}[Serre Duality]\label{GSDuality}
The sheaf of differentials $ \Om{\Ml{2}} $ is a dualizing sheaf on $ \Ml{2} $, i.e. the natural cup product map \[ H^0(\Ml{2},\om{\Ml{2}}{t})\otimes H^1(\Ml{2},\om{\Ml{2}}{-t}\otimes\Om{\Ml{2}})\to H^1(\Ml{2},\Om{\Ml{2}}), \]
is a perfect pairing which is compatible with the $ S_3 $-action.
\end{thm}
\begin{rem} Compatibility with the $ S_3 $-action simply means that for every $ g\in S_3 $, the following diagram commutes
{\small
\begin{align*}
\xymatrix@C=0.4pc{
H^0(\Ml{2},g^*\om{\Ml{2}}{t})\otimes H^1(\Ml{2},g^*\om{\Ml{2}}{-t}\otimes g^*\Om{\Ml{2}}) \ar[r]\ar[d]_g & H^1(\Ml{2},g^*\Om{\Ml{2}})\ar[d]^g\\
H^0(\Ml{2},\om{\Ml{2}}{t})\otimes H^1(\Ml{2},\om{\Ml{2}}{-t}\otimes\Om{\Ml{2}}) \ar[r]& H^1(\Ml{2},\Om{\Ml{2}})}.
\end{align*}
}
But we have made a choice of generators for $\Lambda= H^0(\Ml{2},\om{\Ml{2}}{*})\cong H^0(\Ml{2},g^*\om{\Ml{2}}{*})$, and we have described the $ S_3 $-action on those generators in Lemma \ref{ActionLambda}. If we think of the induced maps $ g:H^*(\Ml{2},g^*\om{\Ml{2}}{*})\to H^*(\Ml{2},\om{\Ml{2}}{*}) $ as a change of basis action of $ S_3 $, Theorem \ref{GSDuality} states that we have a perfect pairing of $ S_3 $-modules. As a consequence, there is an $ S_3 $-module isomorphism
\[ H^1(\Ml{2},\om{\Ml{2}}{*-4})\cong\Hom(\Lambda,\Z_{\sgn}[1/2])=\Lambda^\vee_{\sgn}. \]
(The subscript $ \sgn $ will always denote twisting by the sign representation of $ S_3 $.)
\end{rem}

\begin{proof}
Proposition \ref{M2cohomology} and Lemma \ref{M2Omega} give us explicitly all of the modules involved. Namely, $ H^0(\Proj\Lambda,\Sh{O}(*)) $ is free on the monomials $ \lambda_1^i\lambda_2^j $, for $ i,j\geq 0 $, and $ H^1(\Proj\Lambda,\Sh{O}(*)) $ is free on the monomials $ \frac{1}{\lambda_1^i\lambda_2^j }= \frac{1}{\lambda_1^{i-1}\lambda_2^{j-1} }\frac{1}{\lambda_1\lambda_2 }$, for $ i,j>0 $. Lemma \eqref{M2Omega} gives us in addition that $ H^1(\Proj\Lambda,\Sh{O}(*)\otimes\Om{\Ml{2}}) $ is free on $\frac{1}{\lambda_1^{i-1}\lambda_2^{j-1} }\frac{\eta}{\lambda_1\lambda_2 }$. We conclude that
\[ (\lambda_1^i\lambda_2^j, \frac{\eta}{\lambda_1^{i+1}\lambda_2^{j+1}})\mapsto \frac{\eta}{\lambda_1\lambda_2} \]
is really a perfect pairing.

Moreover, this pairing is compatible with any projective transformation $ \varphi $ of $ \Proj\Lambda $, which includes the $ S_3 $-action as well as change of basis. Any such $ \varphi $ acts on $ H^*(\Proj\Lambda,\Sh{O}(*)) $ by a linear change of variables, and changes $ \eta $ by the determinant $ \det\varphi $. Thus the diagram
{\small
\begin{align*}
\xymatrix@C=0.4pc{
H^0(\Ml{2},\varphi^*\om{\Ml{2}}{t}) \otimes H^1(\Ml{2},\varphi^*\om{\Ml{2}}{-t} \otimes \varphi^*\Om{\Ml{2}}) \ar[r]\ar@<-13ex>[d]_{\varphi}\ar@<6ex>[d]^{\varphi\otimes\det\varphi} & H^1(\Ml{2},\varphi^*\Om{\Ml{2}})\ar[d]^{\det\varphi}\\
H^0(\Ml{2},\om{\Ml{2}}{t}) \otimes H^1(\Ml{2},\om{\Ml{2}}{-t}\otimes\Om{\Ml{2}}) \ar[r]& H^1(\Ml{2},\Om{\Ml{2}})}
\end{align*}
}
commutes. \end{proof}

We explicitly described the induced action on the global sections $ H^0(\Proj\Lambda,\Sh{O}(*)) = \Lambda $ in Lemma \ref{ActionLambda}, and  in \eqref{M2Omega} we have identified $ H^1(\Proj\Lambda,\Om{\Proj\Lambda}) $ with the sign representation $ \Z_{\sgn} $ of $ S_3 $. Therefore, the perfect pairing is the natural map
\[\Lambda\otimes\Lambda_{\sgn}^\vee\to\Z_{\sgn}[1/2].\]

\section{Anderson Duality for $ Tmf(2) $}\label{sec:andersontmf2}

The above Serre duality pairing for $ \Ml{2} $ enables us to compute the homotopy groups of $ Tmf(2) $ as a module over $ S_3 $. We obtain that the $ E_2 $ term of the spectral sequence \eqref{ss:jardine} for $ Tmf(2) $ looks as follows:
\begin{figure}[h]
\centering
\includegraphics[width=\textwidth]{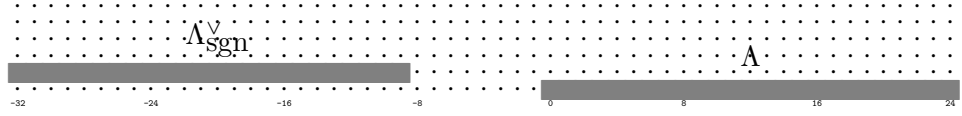}
\caption{Jardine spectral sequence \eqref{ss:jardine} for $\pi_* Tmf(2) $ }\label{Fig:tmf2}
\end{figure}

As there is no space for differentials or extensions, we conclude that \[ \pi_*Tmf(2)=\Lambda\oplus\Sigma^{-9}\Lambda^\vee_{\sgn}. \]
Even better, we are now able to prove self-duality for $ Tmf(2) $.

Since we are working with $ 2 $ inverted everywhere, the Anderson dual of $ Tmf(2) $ is defined by dualizing the homotopy groups as $ \Z[1/2] $-modules, as noted in Remark \ref{rem:Anderson}. Although this duality may deserve the notation $ I_{\Z[1/2]} $, we forbear in the interest of compactness of the notation.

Recall that the Anderson dual of $ Tmf(2) $ is the function spectrum $ F(Tmf(2),\IZ) $, so it inherits an action by $ S_3 $ from the one on $ Tmf(2) $.

\begin{thm}\label{thm:andersontmf2}
The Anderson dual of $ Tmf(2) $ is $ \Sigma^9Tmf(2) $. The inherited $ S_3 $-action on $ \pi_*\IZ Tmf(2) $ corresponds to the action on $ \pi_*\Sigma^9Tmf(2) $ up to a twist by sign.
\end{thm}
\begin{proof}
On the level of homotopy groups, we have the spectral sequence \eqref{ss:AndersonDual} which collapses because each $ \pi_iTmf(2) $ is a free (and dualizable) $ \Z $-module. Thus
\[ \pi_*\IZ Tmf(2)=(\Lambda\oplus\Sigma^{-9}\Lambda^\vee_{\sgn})^\vee,\]
which is isomorphic to $ \Sigma^9(\pi_*Tmf(2))_{\sgn}=\Lambda^\vee\oplus\Sigma^9\Lambda_{\sgn} $, as $ \pi_*Tmf(2) $-modules via a double duality map. Now, $ \IZ Tmf(2) $, being defined as the function spectrum $ F(Tmf(2),\IZ) $, is naturally a $ Tmf(2) $-module, thus a dualizing class $f: S^9\to \IZ Tmf(2) $ extends to an equivalence $\tilde{f}: \Sigma^9Tmf(2)\to\IZ Tmf(2) $. Specifically, let $ f:S^9\to\IZ Tmf(2) $ represent a generator of $ \pi_9\IZ Tmf(2)=\Z[1/2] $, which is also a generator of $ \pi_*\IZ Tmf(2) $ as a $ \pi_*Tmf(2) $-module. Then the composition
\[ \tilde{f}:S^9\wedge Tmf(2) \xrightarrow{f\wedge \Id_{Tmf(2)}} \IZ Tmf(2)\wedge Tmf(2)\xrightarrow{\psi} \IZ Tmf(2), \]
where $ \psi $ is the $ Tmf(2) $-action map, gives the required equivalence. Namely, let $ a $ be an element of $ \pi_*Tmf(2) $; then $ \tilde{f}_*(\Sigma^9 a) =[f]a$, but $ f $ was chosen so that its homotopy class generates $ \pi_*\IZ Tmf(2) $ as a $ \pi_*Tmf(2) $-module.
\end{proof}

\section{Group Cohomology Computations}\label{sec:grpcohomology}

This section is purely technical; for further use, we compute the $S_3$ homology and cohomology of the module $H^*(\Ml{2},\om{\Ml{2}}{*})=\Lambda\oplus\Sigma^{-9}\Lambda^\vee_{\sgn}$, where the action is described in Lemma \ref{ActionLambda}. First we deal with Tate cohomology, and then we proceed to compute the invariants and coinvariants.

\subsection{Tate Cohomology}\label{sec:tatecohomology}

The symmetric group on three letters fits in a short exact sequence
\[
1\to C_3\to S_3\to C_2\to 1,
\]
producing a Lyndon-Hochschield-Serre spectral sequence for the cohomology of $ S_3 $. If $ 2 $ is invertible in the $ S_3 $-module $ M $, the spectral sequence collapses to give that the $ S_3 $ cohomology, as well as the $ S_3 $-Tate cohomology, is computed as the fixed points of the $ C_3 $-analogue
\begin{align*}
H^*(C_3,M)^{C_2}\cong H^*(S_3,M)\\
\tH^*(C_3,M)^{C_2}\cong \tH^*(S_3,M).
\end{align*}
Therefore, it suffices to compute the respective $C_3$-cohomology groups as $C_2$-modules.

To do this, we proceed as in \cite{\GHMR}. Give $A=\Z[1/2][x_0,x_1,x_2]$ the left $S_3$-action as follows: $g\in S_3$ maps $x_i$ to $(-1)^{\sgn g}x_{gi}$. We have a surjection of $S_3$-modules $A\to \Lambda$ given by
\begin{align*}
x_0&\mapsto\lambda_1\\
\sigma(x_0)=x_1&\mapsto\lambda_2-\lambda_1=\sigma(\lambda_1)\\
\sigma^2(x_0)=x_2&\mapsto-\lambda_2=\sigma^2(\lambda_2).
\end{align*}
The kernel of this map is the ideal generated by $\sigma_1=x_0+x_1+x_2$. Therefore, we have a short exact sequence
\begin{equation}\label{AtoLambda}
0\to A\sigma_1\to A\to \Lambda\to 0.
\end{equation}

The orbit under $\sigma$ of each monomial of $A$ has $3$ elements, unless that polynomial is a power of $\sigma_3=x_0x_1x_2$. Therefore, $A$ splits as a sum of a $S_3$-module $ F $ with free $C_3$-action and $\Z[\sigma_3]$ which has trivial $C_3$-action, i.e.
\begin{equation}\label{Adecomposition}
A=F\oplus\Z[\sigma_3].
\end{equation}
Let $ N:A\to H^0(C_3,A) $ be the additive norm map, and let $d$ denote the cohomology class in bidegree $ (0,6) $ represented by $\sigma_3$. Then we have an exact sequence
\[ A\xrightarrow{N} H^*(C_3,A)\to \Z/3[b,d]\to 0, \]
where $ b $ is a cohomology class of bidegree $ (2,0) $. The Tate cohomology of $ A $ is then
\[
\tH^*(C_3,A)\cong H^*(C_3,A)[b^{-1}] \stackrel{\sim}{\rightarrow}\Z/3[b^{\pm 1},d].
\]
The quotient $C_2$-action is given by $\tau(b)=-b$ and $\tau(d)=-d$. Similarly, noting that the degree of $\sigma_1$ is $2$, and $\tau(\sigma_1)=-\sigma_1$, we obtain that the $ C_3 $-cohomology of the module $ A\sigma_1 $ is the same as that of $ A $, with the internal grading shifted by $ 2 $, and the quotient $ C_2 $-action twisted by sign. In other words,
\[
\tH^*(C_3,A\sigma_1)\cong \Sigma^2\left( (\Z_{\sgn}/3)[\tilde b^{\pm 1},\tilde d]\right).
\]
where again $\tilde b$ and $\tilde d$ have bidegrees $(2,0)$ and $(0,6)$ respectively, and the quotient $C_2$-action is described by
\[
\tau:\quad\tilde b^i\tilde d^j\mapsto (-1)^{i+j+1}\tilde b^i\tilde d^j.
\]
Note that $\tH^*(C_3,A)$ and $\tH^*(C_3,A\sigma_1)$ are concentrated in even cohomological degrees. Therefore, the long exact sequence in cohomology induced by \eqref{AtoLambda} breaks up into the exact sequences
\begin{align}
0\to \tH^{2k-1}(C_3,\Lambda)\to \tH^{2k}(C_3,A\sigma_1)\to \tH^{2k}(C_3,A)\to \tH^{2k}(C_3,\Lambda)\to 0.
\end{align}
The middle map in this exact sequence is zero, because it is induced by multiplication by $ \sigma_1 $, which is in the image of the additive norm on $A$. It follows that
\[
\tH^*(C_3,\Lambda)\cong\Z/3[a,b^{\pm 1},d]/(a^2),
\]
where $ a $ is the element in bidegree $ (1,2) $ which maps to $ \tilde{b}\in \tH^2(C_3,A\sigma_1) $. The quotient action by $C_2$ is described as
\begin{align*}
\tau:\quad &a\mapsto a\\
&b\mapsto -b\\
&d\mapsto -d.
\end{align*}
Now it only remains to take fixed points to compute the Tate cohomology of $ \Lambda $ and $ \Lambda_{\sgn} $.
\begin{proposition}\label{prop:TateCohomology}
Denote by $ R $ the graded ring $ \Z/3[a,b^{\pm 2},d^2]/(a^2) $. Then Tate cohomology of the $ S_3 $-modules $ \Lambda $ and $ \Lambda_{\sgn} $ is
\begin{align*}
\tH^*(S_3,\Lambda)&=R\oplus Rbd,\\
\tH^*(S_3,\Lambda_{\sgn})&=Rb\oplus Rd. 
\end{align*}
\end{proposition}

\begin{rem}
The classes $a$ and $bd$ will represent the elements of $ \pi_*Tmf $ commonly known as $\alpha$ and $\beta$, respectively, at least up to a unit.
\end{rem}

\subsection{Invariants}\label{sec:invariants}

We now proceed to compute the invariants $ H^0(S_3,H^*(\Ml{2}, \om{ \Ml{2}} {*} ) )$. The result is summarized in the next proposition.
\begin{proposition}\label{prop:invariants}
The invariants of $ \Lambda $ under the $ S_3 $-action are isomorphic to the ring of modular forms $ MF_*[1/2] $, i.e.
\[ \Lambda^{S_3}= \Z[1/2][c_4,c_6,\Delta]/(1728\Delta-c_4^3-c_6^2).  \]
The twisted invariants module $ \Lambda_{\sgn}^{S_3} $ is a free $ \Lambda^{S_3} $-module on a generator $ d $ of degree $ 6 $.
\end{proposition}

\begin{proof}
Let $\varepsilon\in A$ denote the alternating polynomial $(x_1-x_2)(x_1-x_3)(x_2-x_3)$. Then $\varepsilon^2$ is symmetric, so it must be a polynomial $ g(\sigma_1,\sigma_2,\sigma_3) $ in the elementary symmetric polynomials. Indeed, $ g $ is the discriminant of the polynomial \[(x-x_0)(x-x_1)(x-x_2)=x^3+\sigma_1x^2+\sigma_2x+\sigma_3.\quad \eqref{Ex}\]
The $ C_3 $ invariants in $ A $ are the alternating polynomials in three variables
\[
A^{C_3}=\Z[\sigma_1,\sigma_2,\sigma_3,\varepsilon]/(\varepsilon^2-g).
\]
The quotient action by $ C_2 $ fixes $\sigma_2$ and $\varepsilon$, and changes the sign of $\sigma_1$ and $\sigma_3$. Since $ C_3 $ fixes $ \sigma_1 $, the invariants in the ideal in $ A $ generated by $ \sigma_1 $ are the ideal generated by $ \sigma_1 $ in the invariants $ A^{C_3} $. As $ H^1(C_3,A\sigma_1) =0$, the long exact sequence in cohomology gives a short exact sequence of invariants
\[
\sigma_1 A^{C_3}\to A^{C_3}\to\Lambda^{C_3}\to 0.
\]
Denoting by $ p $ the quotient map $ A^{C_3}\to\Lambda^{C_3} $, we now have
\[
\Lambda^{C_3}\cong A^{C_3}/(\sigma_1)=\Z[p(\sigma_2),p(\sigma_3),p(\varepsilon)]/p(\varepsilon^2+27\sigma_3^2+4\sigma_2^3)
\]
where $\tau$ fixes $p(\sigma_2)$ and $p(\varepsilon)$ and changes the sign of $p(\sigma_3)$. It is consistent with the above computations of Tate cohomology to denote $ \sigma_3 $ and $ p(\sigma_3) $ by $ d $. The invariant quantities are well-known; they are the modular forms of $E_\lambda$ of \eqref{Elambda}, the universal elliptic curve over $\Ml{2}$:
\begin{align}\label{eq:lambdamodular}
\begin{split}
p(\sigma_2)=&-(\lambda_1^2+\lambda_2^2-\lambda_1\lambda_2)=-\frac{1}{16}c_4\\
p(\varepsilon)=&-(\lambda_1+\lambda_2)(2\lambda_1^2+2\lambda_2^2-5\lambda_1\lambda_2)=\frac{1}{32}c_6\\
p(\sigma_3^2)=&d^2=\lambda_1^2\lambda_2^2(\lambda_2-\lambda_1)^2=\frac{1}{16}\Delta.
\end{split}
\end{align}

Hence, $ d $ is a square root of the discriminant $ \Delta $, and since $ 2 $ is invertible, we get that the invariants
\begin{align}\label{eq:invariants}
\begin{split}
\Lambda^{S_3}&=\Z[1/2][p(\sigma_2),p(\sigma_3^2),p(\varepsilon)]/p(\varepsilon^2+27\sigma_3^2+4\sigma_2^3)\\
&=\Z[1/2][c_4,c_6,\Delta]/(1728\Delta-c_4^3+c_6^2)=MF_*
\end{split}
\end{align}
are the ring of modular forms, as expected. Moreover, there is a splitting
$ \Lambda^{C_3}\cong \Lambda^{S_3}\oplus d\Lambda^{S_3}$, giving that
\begin{align}\label{inv:splitting}
\Lambda_{\sgn}^{S_3}=d\Lambda^{S_3}.
\end{align}
\end{proof}

\subsection{Coinvariants and Dual Invariants}\label{sec:coinvariants}

To be able to use Theorem \eqref{GSDuality} to compute homotopy groups, we also need to know the  $ S_3 $-cohomology of the signed dual of $ \Lambda $. For this, we can use the composite functor spectral sequence for the functors $ \Hom_{\Z}(-,\Z) $ and $ \Z\tensor{\Z S_3}(-) $. Since $ \Lambda $ is free over $ \Z $, we get that \[ \Hom_{\Z}(\Z\tensor{\Z S_3}\Lambda_{\sgn},\Z)\cong \Hom_{\Z S_3}(\Z,\Lambda^\vee_{\sgn}),\]and a spectral sequence
\begin{align}\label{ss:dualcohomology}
\Ext^p_{\Z}(H_q(S_3,\Lambda_{\sgn}),\Z)\Rightarrow H^{p+q}(S_3,\Lambda_{\sgn}^\vee).
\end{align}

The input for this spectral sequence is computed in the following lemma.
\begin{lemma}\label{lemma:coinvariants}
The coinvariants of $ \Lambda $ and $ \Lambda_{\sgn} $ under the $ S_3$ action are
\begin{align*}
H_0(S_3,\Lambda)=(3,c_4,c_6)\oplus ab^{-1}d \Z/3[\Delta]\\
H_0(S_3,\Lambda_{\sgn})=d(3,c_4,c_6)\oplus ab^{-1} \Z/3[\Delta],
\end{align*}
where $ (3,c_4,c_6) $ is the ideal of the ring $ \Lambda^{S3}=MF_* $ of modular forms generated by $ 3,c_4 $ and $ c_6 $, and $ d(3,c_4,c_6) $ is the corresponding submodule of the free $ \Lambda^{S_3}$-module generated by $ d $.
\end{lemma}
\begin{proof}
We use the exact sequence 
\begin{align}\label{seq:normM}
0\to \tH^{-1}(S_3,M)\to H_0(S_3,M)\xrightarrow{N} H^0(S_3,M)\to \tH^0(S_3,M)\to 0.
\end{align}
For $ M=\Lambda $, this is
\[ 0\to ab^{-1}d\Z/3[\Delta]  \to H_0(S_3,\Lambda )\xrightarrow{N} \Z[c_4,c_6,\Delta]/(\sim) \xrightarrow{\pi} \Z/3[\Delta]\to 0, \]
where the rightmost map $ \pi $ sends $ c_4 $ and $ c_6 $ to zero, and $ \Delta $ to $ \Delta $. Hence its kernel is the ideal $ (3,c_4,c_6) $, which is a free $ \Z $-module so we have a splitting as claimed.

Similarly, for $ M=\Lambda_{\sgn} $, the exact sequence \eqref{seq:normM} becomes
\[ 0\to ab^{-1}\Z/3[\Delta]  \to H_0(S_3,\Lambda_{\sgn})\xrightarrow{N} d\Z[c_4,c_6,\Delta]/(\sim) \xrightarrow{d\pi} d\Z/3[\Delta]\to 0. \]
The kernel of $ d\pi $ is the ideal $ d(3,c_4,c_6) $, and the result follows.
\end{proof}

\begin{cor}\label{prop:dualinvariants}
The $ S_3 $-invariants of the dual module $ \Lambda^\vee $ are the module dual to the ideal $ (3,c_4, c_6) $, and the $ S_3 $-invariants of the dual module $ \Lambda_{\sgn}^\vee $ are the module dual to the ideal $ d(3,c_4,c_6) $.
\end{cor}
\begin{proof}
In view of the above spectral sequence \eqref{ss:dualcohomology}, to compute the invariants it suffices to compute the coinvariants, which we just did in Lemma \ref{lemma:coinvariants}, and dualize.
\end{proof}

We need one more computational result crucial in the proof of the main Theorem \ref{MainThm}.

\begin{proposition}\label{prop:e2shift}
There is an isomorphism of modules over the cohomology ring $H^{*}(S_{3}, \pi_{*}Tmf(2))$
\[ H^*(S_3,\pi_*\IZ Tmf(2))\cong H^*(S_3,\pi_*\Sigma^{21}Tmf(2)). \]
\end{proposition}
\begin{proof}
We need to show that $ H^*(S_3,\Sigma^9\Lambda_{\sgn}\oplus\Lambda^\vee )$ is a shift by $ 12 $ of $ H^*(S_3,\Sigma^9\Lambda\oplus\Lambda^\vee_{\sgn}) $. First of all, we look at the non-torsion elements. Putting together the results from equation \eqref{inv:splitting} and Corollary \ref{prop:dualinvariants} yields
\[ H^0(S_3,\Sigma^9\Lambda_{\sgn}\oplus\Lambda^\vee)= d H^0(S_3,\Sigma^9\Lambda\oplus \Lambda_{\sgn}^\vee), \]
and indeed we shall find that the shift in higher cohomology also comes from multiplication by the element $ d $ (of topological grading $ 12 $).

Now we look at the higher cohomology groups, computed in Proposition \ref{prop:TateCohomology}. Identifying $ (b^{-1})^\vee $ with $ b $, we obtain, in positive cohomological grading
\begin{align*}
H^*(S_3,\pi_*\IZ Tmf(2))=\Sigma^9 H^*(S_3,\Lambda_{\sgn})\oplus H^*(\Lambda^\vee)\\
= \Z/3[b^2,\Delta]\langle \Sigma^9b,\Sigma^9ab,\Sigma^9b^2d,\Sigma^9ad \rangle\\
\oplus \Z/3[b^2,\Delta]/(\Delta^\infty)\langle b^2\Delta, a^\vee b^2\Delta, bd, a^\vee bd \rangle,
\end{align*}
which we are comparing to
\begin{align*}
\Sigma^{21} H^*(S_3,\pi_* Tmf(2))= \Sigma^9d H^*(S_3,\Lambda)\oplus d H^*(S_3,\Lambda^\vee_{\sgn})\\
= \Z/3[b^2,\Delta]\langle \Sigma^9b^2d,\Sigma^9ad,\Sigma^9b\Delta,\Sigma^9 ab\Delta \rangle\\
\oplus \Z/3[b^2,\Delta]/(\Delta^\infty)\langle bd\Delta,a^\vee bd\Delta, b^2\Delta, a^\vee b^2\Delta \rangle.
\end{align*}
Everything is straightforwardly identical, except for the match for the generators $ \Sigma^9b,\Sigma^9ab \in \Sigma^9 H^*(S_3,\Lambda_{\sgn}) $ which have cohomological gradings $ 2 $ and $ 3 $, and topological gradings $ 7 $ and $ 10 $ respectively. On the other side of the equation we have generators $  a^\vee bd, bd \in H^*(S_3,\Lambda^\vee)=\Sigma^9 H^*(S_3,H^1(\Ml{2},\om{}{*})) $, whose cohomological gradings are $ 2 $ and $ 3 $, and topological $ 7 $ and $ 10 $ respectively. Identifying these elements gives an isomorphism which is compatible with multiplication by $ a,b,d $.
\end{proof}

\subsection{Localization}

We record the behavior of our group cohomology rings when we invert a modular form; in Section \ref{sec:htpyfixedpoints} we will be inverting $ c_4 $ and $ \Delta $.

\begin{proposition}
Let $ M $ be one of the modules $ \Lambda $, $ \Lambda_{\sgn} $; the ring of  modular forms $ MF_*=\Lambda^{S_3} $ acts on $ M $. Let $ m\in MF_* $, and let $ M[m^{-1}] $ be the module obtained from $ M $ by inverting the action of $ m $. Then \[ H^*(S_3,M[m^{-1}])\cong H^*(S_3,M)[m^{-1}]. \]
\end{proposition}
\begin{proof}
Since the ring of modular form is $ S_3 $-invariant, the group $ S_3 $ acts $ MF_* $-linearly on $ M $; in fact, $ S_3 $ acts $ \Z[m^{\pm 1}] $-linearly on $ M $, where $ m $ is our chosen modular form. By Exercise 6.1.2 and Proposition 3.3.10 in \cite{\Weibel}, it follows that 
\begin{align*}
 H^*(S_3,M[m^{-1}])&=\Ext^*_{Z[m^{\pm 1}][S_3]}(\Z[m^{\pm 1}], M[m^{-1}])\\
 &=\Ext^*_{\Z[m][S_3]}(\Z[m],M)[m^{-1}]=H^*(S_3,M)[m^{-1}].
 \end{align*}
\end{proof}

Note that if $ M $ is one of the dual modules $ \Lambda^\vee $ or $ \Lambda_{\sgn}^\vee $, the elements of positive degree (i.e. non-scalar elements) in the ring of modular forms $ MF_* $ act on $ M $ nilpotently. Therefore $ M[m^{-1}]=0 $ for such an $ m $.  Moreover, for degree reasons, $c_4 a =0=c_4 b$, and we obtain the following result.
\begin{proposition}\label{prop:c4action}
The higher group cohomology of $ S_3 $ with coefficients in $ \pi_*Tmf(2)[c_4^{-1}] $ vanishes, and
\begin{align*} 
&H^*(S_3,\pi_*Tmf(2)[c_4^{-1}]) = H^0(S_3,\Lambda)[c_4^{-1}]= MF_*[c_4^{-1}].
\end{align*}
Inverting $ \Delta $ has the effect of annihilating the cohomology that comes from the negative homotopy groups of $ Tmf(2) $; in other words,
\begin{align*}
&H^*(S_3,\pi_*TMF(2))=H^*(S_3,\pi_*Tmf(2)[\Delta^{-1}])= H^*(S_3,\Lambda)[\Delta^{-1}].
\end{align*}
\end{proposition}

\section{Homotopy Fixed Points}

In this section we will use the map $ q:\Ml{2}\to\M[1/2] $ and our knowledge about $ Tmf(2) $ from the previous sections to compute the homotopy groups of $ Tmf $, in a way that displays the self-duality we are looking for. Economizing the notation, we will write $ \M $ to mean $ \M[1/2] $ throughout.

\subsection{Homotopy Fixed Point Spectral Sequence}\label{sec:htpyfixedpoints}

We will use Theorem \ref{thm:tmffixedpt} to compute the homotopy groups of $ Tmf $ via the homotopy fixed point spectral sequence
\begin{align}\label{ss:htpyfixed}
H^*(S_3,\pi_*Tmf(2))\Rightarrow \pi_*Tmf.
\end{align}
We will employ two methods of calculating the $ E_2 $-term of this spectral sequence: the first one is more conducive to computing the differentials, and the second is more conducive to understanding the duality pairing.

\subsubsection*{Method One}
The moduli stack $ \M $ has an open cover by the substacks $ \M^0=\M[\Delta^{-1}] $ and $ \M[c_4^{-1}] $, giving the cube of pullbacks	
\[ \xymatrix@C=3pt@R=6pt{
& \M^0[c_4^{-1}] \ar[rr]\ar[dd]& &\M^0\ar[dd]\\
\Ml{2}^0[c_4^{-1}] \ar[ru]\ar[rr]\ar[dd] && \Ml{2}^0\ar[ru]\ar[dd]\\
&\M[c_4^{-1}]\ar[rr] &&\M.\\
\Ml{2}[c_4^{-1}]\ar[ru]\ar[rr] && \Ml{2}\ar[ru]
} \]
Since $ \Delta,c_4 $ are $ S_3 $-invariant elements of $ H^*(\Ml{2},\omega^*) $, the maps in this diagram are compatible with the $ S_3 $-action. 
Taking global sections of the front square, we obtain a cofiber sequence \[ Tmf(2) \to TMF(2)\vee Tmf(2)[c_4^{-1}] \to TMF(2)[c_4^{-1}],\] compatible with the $ S_3 $-action.
Consequently, there is a cofiber sequence of the associated homotopy fixed point spectral sequences, converging to the cofiber sequence from the rear pullback square of the above diagram
\begin{align}\label{cof:tmfc4del}
 Tmf \to TMF\vee Tmf[c_4^{-1}] \to TMF[c_4^{-1}].
 \end{align} 
We would like to deduce information about the differentials of the spectral sequence for $ Tmf $ from the others. According to Proposition \ref{prop:c4action}, we know that the spectral sequences for $ Tmf[c_4^{-1}] $ and $ TMF[c_4^{-1}] $ collapse at their $ E_2 $ pages and that all torsion elements come from $ TMF $. From Proposition \ref{prop:TateCohomology}, we know what they are. 

As the differentials in the spectral sequence \eqref{ss:htpyfixed} involve the torsion elements in the higher cohomology groups $ H^*(S_3,\pi_*Tmf(2)) $, they have to come from the spectral sequence for $ TMF $\footnote{Explicitly, the cofiber sequence \eqref{cof:tmfc4del} gives a commutative square of spectral sequences just as in the proof of Theorem \ref{MainThm} below, which allows for comparing differentials.}, where they are (by now) classical. They are determined by the following lemma, which we get from \cite{\Bauer} or \cite{\Rezk512}.
\begin{lemma}
The elements $ \alpha $ and $ \beta $ in $ \pi_*S_{(3)} $ are mapped to $ a $ and $ bd $ respectively under the Hurewicz map $ \pi_*S_{(3)}\to\pi_*TMF $.
\end{lemma}
Hence, $ d_5(\Delta)=\alpha\beta^2 $, $ d_9(\alpha\Delta) =\beta^5$, and the rest of the pattern follows by multiplicativity.

The aggregate result is depicted in the chart of Figure \ref{fig:HtpyFixed}.

\begin{figure}[p]
\centering
\includegraphics[angle=90, height=.9\textheight]{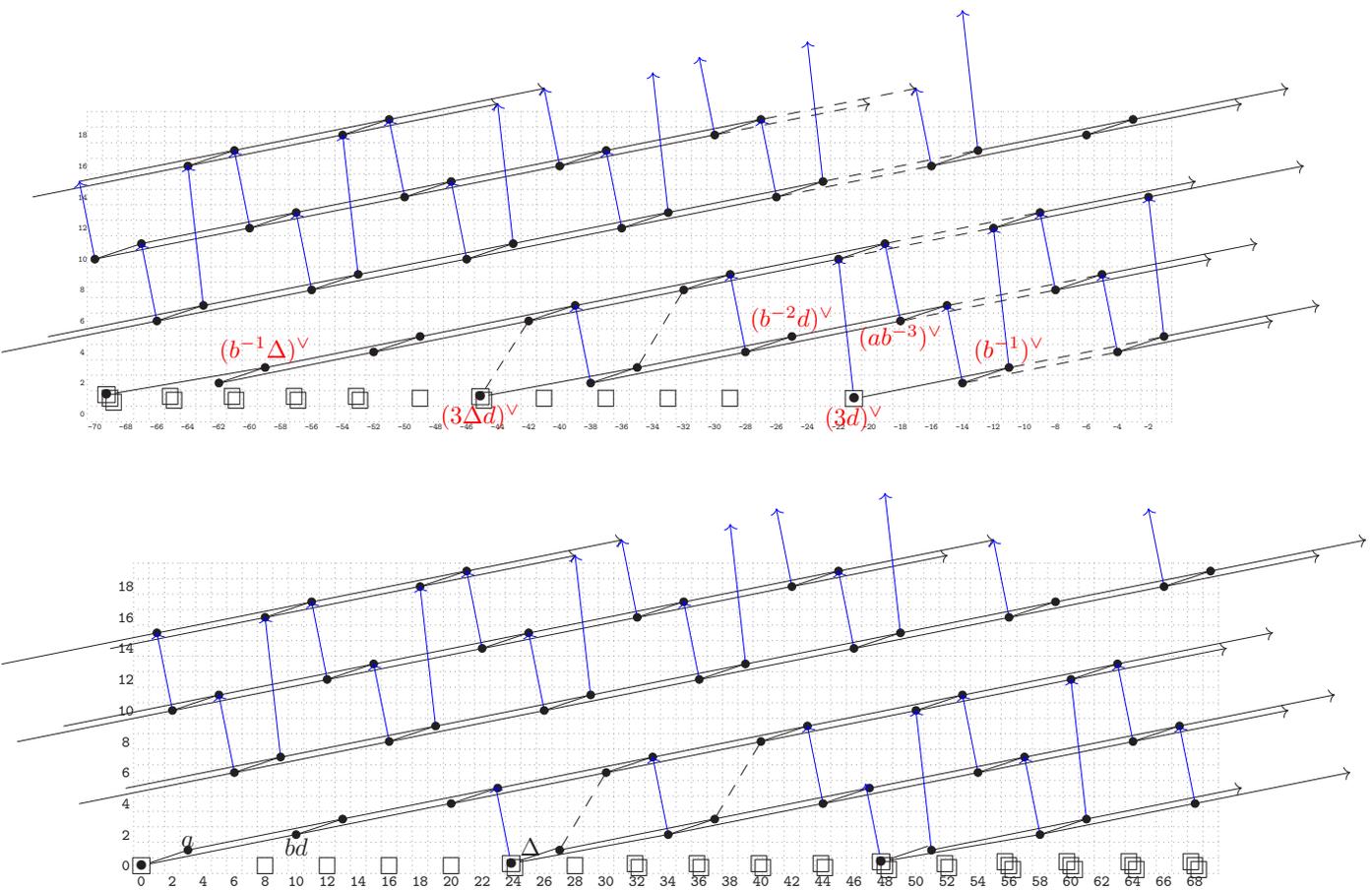}
\caption{Homotopy fixed point spectral sequence \eqref{ss:htpyfixed} for $ \pi_* Tmf $}\label{fig:HtpyFixed}
\end{figure}

\subsubsection*{Method Two}
On the other hand, we could proceed using our Serre duality for $ \Ml{2} $ and the spectral sequence \eqref{ss:dualcohomology}. The purpose is to describe the elements below the line $ t=0 $ as elements of $ H^*(S_3,\Lambda_{\sgn}^\vee) $. Indeed, according to the Serre duality pairing from Theorem \ref{GSDuality}, we have an isomorphism \[ H^s(S_3,H^1(\Ml{2},\om{}{t-4}))\cong H^s(S_3,\Lambda_{\sgn}^\vee), \]
and the latter is computed in Corollary \ref{prop:dualinvariants} via the collapsing spectral sequence \eqref{ss:dualcohomology} and Lemma \ref{lemma:coinvariants}
\[ \Ext^p_{\Z}(H_q(S_3,\Lambda_{\sgn}),\Z)\Rightarrow H^{p+q}(S_3,\Lambda_{\sgn}^\vee). \]
In Section \ref{sec:grpcohomology} we computed the input for this spectral sequence. The coinvariants are given as 
\[ H_0(S_3,\Lambda_{\sgn})=\Z/3[\Delta]ab^{-1}\oplus d(3,c_4,c_6) \]
by Lemma \ref{lemma:coinvariants}, and the remaining homology groups are computed via the Tate cohomology groups. Namely, for $ q\geq 1 $,
\[ H_q(S_3,\Lambda_{\sgn})\cong \tH^{-q-1}(S_3,\Lambda_{\sgn}) \] 
which, according to Proposition \ref{prop:TateCohomology}, equals to the part of cohomological degree $ (-q-1) $ in the Tate cohomology of $ \Lambda_{\sgn} $, which is $ Rb\oplus Rd $, where $ R=\Z/3 [a, b^{\pm 2}, \Delta]/ (a^2) $. Recall, the cohomological grading of $ a $ is one, that of $ b $ is two, and $ \Delta $ has cohomological grading zero.

In particular, we find that the invariants $ H^0(S_3,\Lambda_{\sgn}^\vee) $ are the module dual to the ideal $ d(3,c_4,c_6) $. This describes the negatively graded non-torsion elements. For example, the element dual to $ 3d $ is in bidegree $ (t,s)=(-10,1) $. We can similarly describe the torsion elements as duals. If $ X $ is a torsion abelian group, let $ X^\vee $ denote $ \Ext^1_{\Z}(X,\Z) $, and for $ x\in X $, let $ x^\vee $ denote the element in $ X^\vee $ corresponding to $ x $ under an isomorphism $ X\cong X^\vee $.  For example, $ (ab^{-1})^\vee $ is in bidegree $ (-6,2) $, and the element corresponding to $ b^{-2}d $ lies in bidegree $ (-10,5) $.

\subsection{Duality Pairing}

Consider the non-torsion part of the spectral sequence \eqref{ss:htpyfixed} for $\pi_*Tmf $. According to Lemma \ref{lemma:coinvariants}, on the $ E_2 $ page, it is $ MF_*\oplus \Sigma^{-9}d^\vee(3,c_4,c_6)^\vee $. Applying the differentials only changes the coefficients of various powers of $ \Delta $. Namely, the only differentials supported on the zero-line are $ d_5(\Delta^{m}) $ for non-negative integers $ m $ not divisible by $ 3 $, which hit a corresponding class of order $ 3 $. Therefore, $ 3^\epsilon\Delta^{m} $ is a permanent cycle, where $ \epsilon $ is zero if $ m $ is divisible by $ 3 $ and one otherwise. In the negatively graded part, only $  (3\Delta^{3k}d)^\vee $, for non-negative $ k $, support a differential $ d_9 $ and hit a class of order 3, thus $ (3^\epsilon\Delta^m d)^\vee $ are permanent cycles, for $ \epsilon $ as above. The pairing at $ E_\infty $ is thus obvious: $3^\epsilon\Delta^{m}c_4^ic_6^j$ and $ (3^\epsilon\Delta^m dc_4^ic_6^j)^\vee $ match up to the generator of $ \Z $ in $ \pi_{-21}Tmf $.

The pairing on torsion depends even more on the homotopy theory, and interestingly not only on the differentials, but also on the exotic multiplications by $ \alpha $. The non-negative graded part is $ \Z[\Delta^3] $ tensored with the pattern in Figure \ref{fig:TorPos}, whereas the negative graded part is $ (\Z[\Delta^3]/\Delta^\infty )\frac{d^\vee}{\lambda_1\lambda_2}$ tensored with the elements depicted in Figure \ref{fig:TorNeg}. 

\begin{figure}[h]
\centering
\includegraphics[width=\textwidth]{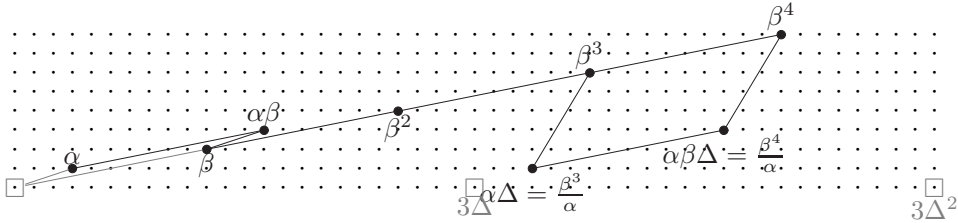}
\caption{Torsion in positive degrees}\label{fig:TorPos}
\end{figure}

Everything pairs to $ \alpha^\vee\beta^5(\Delta^2)^\vee $, which is $ d_9 (\frac{(3d)^\vee}{\lambda_1\lambda_2}) $, i.e. the image under $ d_9 $ of $ 1/3 $ of the dualizing class. Even though $ \alpha^\vee\beta^5(\Delta^2)^\vee $ is zero in the homotopy groups of $ Tmf $, the corresponding element in the homotopy groups of the $ K(2) $-local sphere is nontrivial \cite{\HKM}, thus it makes sense to talk about the pairing.

\begin{figure}[h]
\centering
\includegraphics[width=\textwidth]{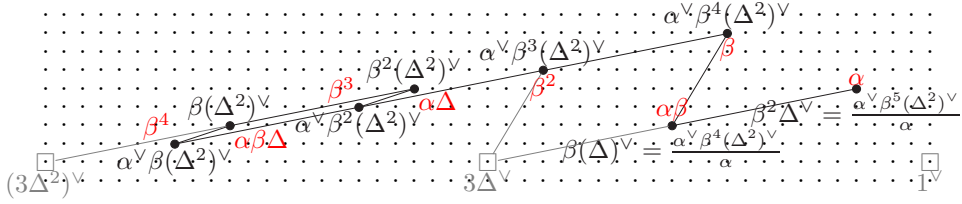}
\caption{Torsion in negative degrees}\label{fig:TorNeg}
\end{figure}

\section{The Tate Spectrum}

In this section we will relate the duality apparent in the homotopy groups of $ Tmf\simeq Tmf(2)^{hS_3}$ to the vanishing of the associated Tate spectrum. The objective is to establish the following:
\begin{thm}\label{thm:norm}
The norm map $ Tmf(2)_{hS_3}\to Tmf(2)^{hS_3} $ is an equivalence.
\end{thm}

A key role is played by the fact that $ S_3 $ has periodic cohomology.

Generalized Tate cohomology was first introduced by Adem-Cohen-Dwyer in \cite{MR1022669}. However, it was Greenlees and May who generalized and improved the theory, and, more importantly, developed excellent computational tools in \cite{\GreenleesMay}. In this section, we shall summarize the relevant results from \cite{\GreenleesMay} and apply them to the problem at hand.

Suppose $ k $ is a spectrum with an action by a finite group $ G $; in the terminology of equivariant homotopy theory, this is known as a \emph{naive} $ G $-spectrum. There is a norm map \cite[5.3]{\GreenleesMay}, \cite[II.7.1]{\LewisMaySteinberger} from the homotopy orbit spectrum $ k_{hG} $ to the homotopy fixed point spectrum $ k^{hG} $ whose cofiber we shall call the \emph{Tate spectrum} associated to the $ G $-spectrum $ k $, and for simplicity denote it by $ k^{tG} $
\begin{align}\label{seq:normfixed}
k_{hG}\to k^{hG}\to k^{tG}.
\end{align}
According to \cite[Proposition 3.5]{\GreenleesMay}, if $ k $ is a ring spectrum, then so are the associated homotopy fixed point and Tate spectra, and the map between them is a ring map.

We can compute the homotopy groups of each of the three spectra in \eqref{seq:normfixed} using the Atiyah-Hirzebruch-type spectral sequences $ \check{E_*}, E_*,\hat{E}_* $ \cite[Theorems 10.3, 10.5, 10.6]{\GreenleesMay}
\begin{align*}
\check{E}_2^{p,q}=H_{-p}(G,\pi_q k) &\Rightarrow \pi_{q-p} k_{hG}\\
E_2^{p,q}=H^p(G, \pi_q k) &\Rightarrow \pi_{q-p} k^{hG}\\
\hat{E}_2^{p,q}=\tH^p(G, \pi_q k) &\Rightarrow \pi_{q-p} k^{tG},
\end{align*}
which are conditionally convergent. As these spectral sequences can be constructed by filtering $ EG $, the first two are in fact the homotopy fixed point and homotopy orbit spectral sequences. In the case when $ k=Tmf(2) $, the first two are half-plane spectral sequences, whereas the third one is in fact a full plane spectral sequence. Moreover, the last two are spectral sequences of differential algebras.

The norm cofibration sequence \eqref{seq:normfixed} relates these three spectral sequences by giving rise to maps between them, which on the $ E_2 $-terms are precisely the standard long exact sequence in group cohomology:
\begin{align}\label{les:tate}
\cdots \to H_{-p}(G,M)\to H^p(G,M)\to \tH^p(G,M)\to H_{-p-1}(G,M)\to\cdots
\end{align}
The map of spectral sequences $ E_*\to \hat{E}_* $ is compatible with the differential algebra structure, which will be important in our calculations as it will allow us to determine the differentials in the Tate spectral sequence (and then further in the homotopy orbit spectral sequence).

\begin{figure}[p]
\centering
\includegraphics[height=.9\textheight]{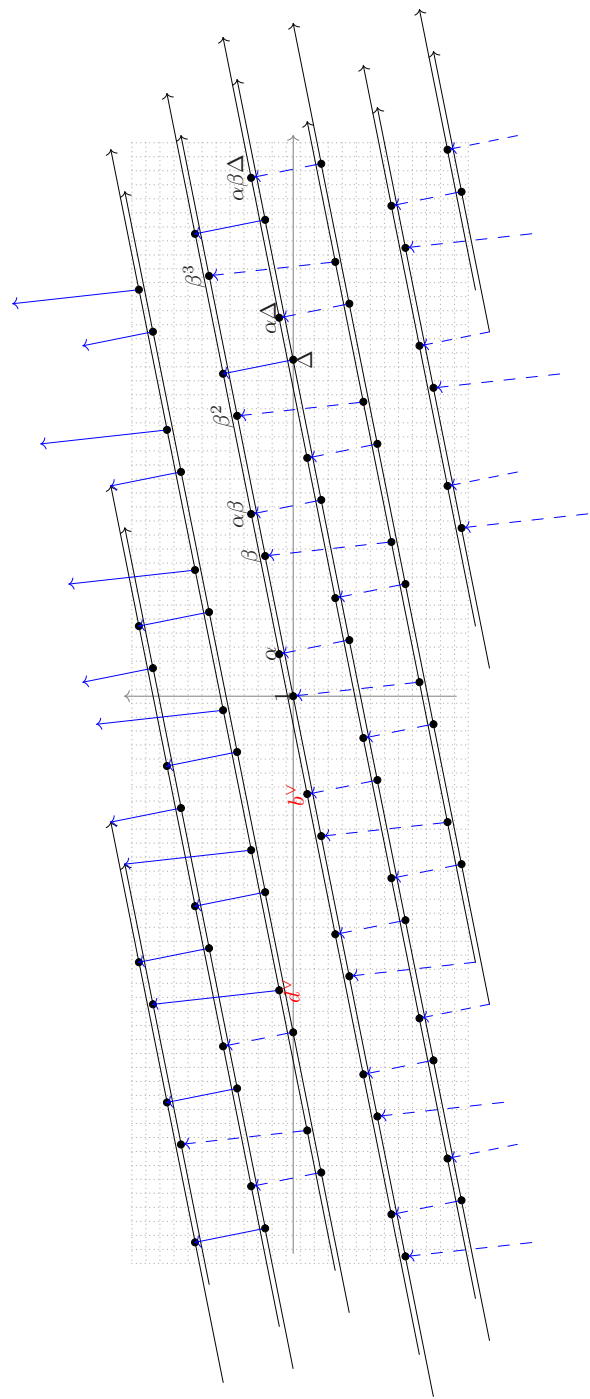}
\caption{Tate spectral sequence \eqref{ss:tate} for $ \pi_* Tmf(2)^{tS_3} $}\label{fig:tate}
\end{figure}

\begin{proof}[Proof of Theorem \ref{thm:norm} ]
We prove that the norm is an equivalence by showing that the associated Tate spectrum is contractible, using the above Tate spectral sequence for the case of $ k=Tmf(2) $ and $ G=S_3 $
\begin{align}\label{ss:tate}
\tH^p (S_3,\pi_{2t-q} Tmf(2)) = \tH^p(S_3,H^q(\Ml{2},\om{}{t}))\Rightarrow \pi_{2t-p-q} Tmf(2)^{tS_3}.
\end{align}
The $ E_2 $-page is the Tate cohomology
\[ \tH^*(S_3,\Lambda\oplus \Sigma^{-9}\Lambda_{\sgn}^\vee), \]
which we computed in Proposition \ref{prop:TateCohomology} to be \[\displaystyle{ R\oplus Rbd\oplus \frac{\eta d^\vee}{\lambda_1\lambda_2} R^\vee \oplus \frac{\eta b^\vee}{\lambda_1\lambda_2} R^\vee}, \]
where $ R=\Z/3[a,b^{\pm 2},\Delta]/(a^2) $, and $ R^\vee=\Ext^1_{\Z}(R,\Z) $.
Comparing the two methods for computing the $ E_2 $-page of the homotopy fixed point spectral sequence \eqref{ss:htpyfixed} identifies $ \frac{\eta d^\vee}{\lambda_1\lambda_2} $ with $ \frac{\alpha}{\Delta} $ and $ \frac{\eta b^\vee}{\lambda_1\lambda_2} $ with $ \frac{\alpha}{\beta} $. Further, we can identify $ b^\vee $ with $ b^{-1} $ and similarly for $ \Delta $, as it does not change the ring structure, and does not cause ambiguity. We obtain
\begin{align*}
&\frac{\eta d^\vee}{\lambda_1\lambda_2} R^\vee =\frac{\alpha}{\Delta} R^\vee =\frac{1}{\Delta}\Z/3[a,b^{\pm 2},\Delta^{-1}]/(a^2)\\
&\frac{\eta b^\vee}{\lambda_1\lambda_2}R^\vee = \frac{\alpha}{\beta}R^\vee = \frac{b}{d}\Z/3[a,b^{\pm 2},\Delta^{-1}]/(a^2)=\frac{\beta}{\Delta}\Z/3[a,b^{\pm 2},\Delta^{-1}]/(a^2).
\end{align*}
Summing all up, we get the $ E_2 $-page of the Tate spectral sequence
\begin{align}\label{E2tate}
\tH^*(S_3,\pi_*Tmf(2)) = \Z/3[\alpha,\beta^{\pm 1},\Delta^{\pm 1}]/(\alpha^2),
\end{align}
depicted in Figure \ref{fig:tate}.

 The compatibility with the homotopy fixed point spectral sequence implies that $ d_5(\Delta)=\alpha\beta^2 $ and $ d_9(\alpha\Delta^2)=\beta^5 $; by multiplicativity, we obtain a differential pattern as showed below in Figure \ref{fig:tate}. From the chart we see that the tenth page of the spectral sequence is zero, and, as this was a conditionally convergent spectral sequence, it follows that it is strongly convergent, thus the Tate spectrum $ Tmf(2)^{tS_3} $ is contractible.
\end{proof}

\subsection{Homotopy Orbits}

\begin{figure}[p]
\centering
\includegraphics[angle=90,height=.9\textheight]{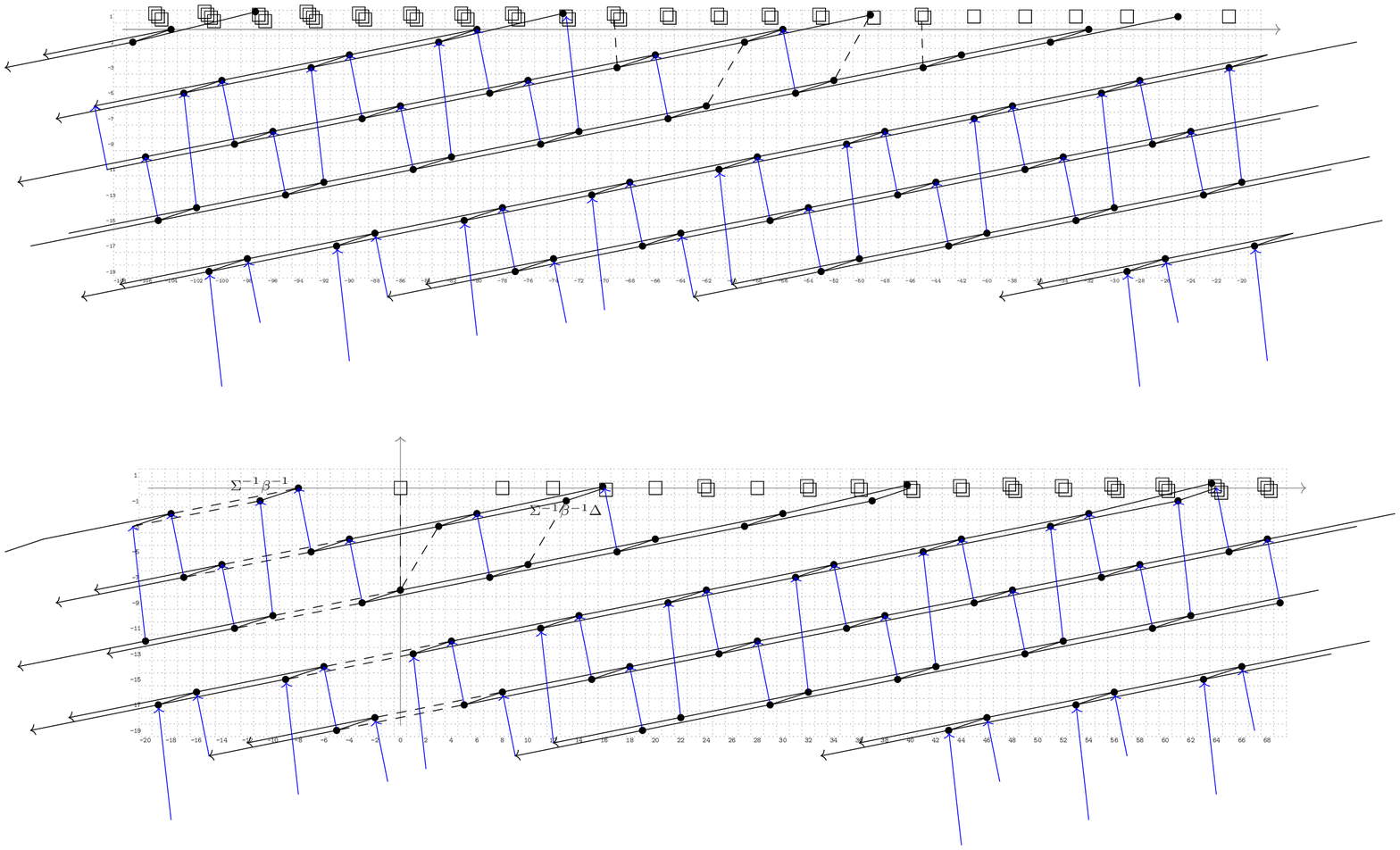}
\caption{Homotopy orbit spectral sequence \eqref{ss:htpyorbit} for $ \pi_* Tmf $}\label{fig:htpyorbit}
\end{figure}

As a corollary of the vanishing of the Tate spectrum $ Tmf(2)^{tS_3} $, we fully describe the homotopy orbit spectral sequence
\begin{align}\label{ss:htpyorbit}
H_{s}(S_3,\pi_t Tmf(2))\Rightarrow \pi_{t+s}Tmf(2)_{hS_3}=\pi_{t+s}Tmf.
\end{align}
From \eqref{E2tate}, we obtain the higher homology groups as well as the differentials. If $ S $ denotes the ring $\Z/3[\beta^{-1},\Delta^{\pm 1}] $,
\[ \bigoplus_{s>0}H_s(S_3,\pi_t Tmf(2))= \Sigma^{-1}( \beta^{-1} S\oplus \alpha\beta^{-2} S).\]
(The suspension shift is a consequence of the fact that the isomorphism comes from the coboundary map $ \tH^*\to H_{-*-1} $.)
The coinvariants are computed in Lemma \ref{lemma:coinvariants}. The spectral sequence is illustrated in Figure \ref{fig:htpyorbit}, with the the topological grading on the horizontal axis and for consistency, the cohomological on the vertical axis.

\section{Duality for $ Tmf $}

In this section we finally combine the above results to arrive at self-duality for $ Tmf $. The major ingredient in the proof is Theorem \ref{thm:norm}, which gives an isomorphism between the values of a right adjoint (homotopy fixed points) and a left adjoint (homotopy orbits). This situation often leads to a Grothendieck-Serre-type duality, which in reality is a statement that a functor (derived global sections) which naturally has a left adjoint (pullback) also has a right adjoint \cite{\FauskHuMay}.

Consider the following chain of equivalences involving the Anderson dual of $ Tmf $
\begin{align*}
\IZ Tmf &= F(Tmf,\IZ) \leftarrow F(Tmf(2)^{hS_3},\IZ) \to F(Tmf(2)_{hS_3},\IZ) \\
&\simeq F(Tmf(2),\IZ)^{hS_3}\simeq(\Sigma^9Tmf(2)_{\sgn})^{hS_3},
\end{align*}
which implies a homotopy fixed point spectral sequence converging to the homotopy groups of $ \IZ Tmf $. From our calculations in Section \ref{sec:grpcohomology} made precise in Proposition \ref{prop:e2shift}, the $ E_2 $-term of this spectral sequence is isomorphic to the $ E_2 $-term for the homotopy fixed point spectral sequence for $ Tmf(2)^{hS_3} $, shifted by $ 21 $ to the right. A shift of $ 9 $ comes from the suspension, and an additional shift of $ 12 $ comes from the twist by sign (which is realized by multiplication by the element $ d $ whose topological degree is $ 12 $). It is now plausible that $ \IZ Tmf $ might be equivalent to $ \Sigma^{21}Tmf $; it only remains to verify that the differential pattern is as desired. To do this, we use methods similar to the comparison of spectral sequences in \cite{MR604321} and in the algebraic setting,  \cite{\Deligne}: a commutative square of spectral sequences, some of which collapse, allowing for the tracking of differentials.

\begin{thm}\label{MainThm}
The Anderson dual of $ Tmf[1/2] $ is $ \Sigma^{21} Tmf[1/2]$.
\end{thm}
\begin{rem}
Here again Anderson duals are taken in the category of spectra with $ 2 $ inverted, and $ 2 $ will implicitly be inverted everywhere  in order for the presentation to be more compact. In particular, $ \Z $ will denote $ \Z[1/2] $, and $ \Q/\Z $ will denote $ \Q/\Z[1/2] $.
\end{rem}
\begin{proof}
For brevity, let us introduce the following notation: for $ R $ any of $ \Z,\Q,\Q/\Z $, we let $ A^\bullet_{R} $ be the cosimplicial spectrum $ F(ES_{3+}\smsh{S_3} Tmf(2),\IR) $. In particular we have that \[ A^h_R=F\bigl( (S_3)_+^{\wedge (h+1)}\smsh{S_3}Tmf(2),\IR\bigr). \]
Then the totalization $ \Tot A^\bullet_{\Z}\simeq \IZ Tmf $ is equivalent to the fiber of the natural map $ \Tot A^\bullet_{\Q} \to \Tot A^\bullet_{\Q/\Z}  $. In other words, we are looking at the diagram
\[\xymatrix{
A^0_{\Q}\ar@2[r]\ar[d] & A^1_{\Q}\ar@3[r]\ar[d] &A^2_{\Q}\ar[d] \cdots\\
A^0_{\Q/\Z}\ar@2[r] & A^1_{\Q/\Z}\ar@3[r] &A^2_{\Q/Z}\cdots\\
}\]
and the fact that totalization commutes with taking fibers gives us two ways to compute the homotopy groups of $ \IZ Tmf $. Taking the fibers first gives rise to the homotopy fixed point spectral sequence whose differentials we are to determine: Each vertical diagram gives rise to an Anderson duality spectral sequence \eqref{ss:AndersonDual}, which collapses at $ E_2 $, as the homotopy groups of each $ \displaystyle{ \bigl((S_3)_+^{\wedge (h+1)}\smsh{S_3}Tmf(2)\bigr) }$ are free over $ \Z $. On the other hand, assembling the horizontal direction first gives a map of the $ \Q $ and $ \Q/\Z $-duals of the homotopy fixed point spectral sequence for the $ S_3$-action on $ Tmf(2) $; this is because $ \Q $ and $ \Q/\Z $ are injective $ \Z $-modules, thus dualizing is an exact functor.

Let $ R^\bullet $ denote the standard injective resolution of $ \Z $, namely $ R^0=\Q $ and $ R^1=\Q/\Z $ related by the obvious quotient map. Then, schematically, we have a diagram of $ E_1 $-pages
\[\xymatrix{
\Hom_{\Z}\big(\Z[S_3]^{\otimes(h+1)}\tensor{\Z[S_3]}\pi_t Tmf(2), R^v\big) \ar@2[d]_-*+[o][F-]{A} \ar@2[r]^-*+[o][F-]{B} &\Hom_{\Z}\big(\pi_{t+h}Tmf(2)_{hS_3},R^v\big)\ar@2[d]^-*+[o][F-]{D}\\
\Hom_{\Z}\big(\Z[S_3]^{\otimes(h+1)}\tensor{\Z[S_3]}\pi_t Tmf(2), \Z\big)\ar@2[r]^-*+[o][F-]{C} & \pi_{-t-h-v} \IZ Tmf(2)_{hS_3}.
}\]
The spectral sequence $ A $ collapses at $ E_2 $, and $ C $ is the homotopy fixed point spectral sequence that we are interested in: its $ E_2 $ page is the $ S_3 $-cohomology of the $ \Z $-duals of $ \pi_* Tmf(2) $, which are the homotopy groups of the Anderson dual of $ Tmf(2) $
\begin{align*}
&H^* \Hom_{\Z}\big(\Z[S_3]^{\otimes(h+1)}\tensor{\Z[S_3]}\pi_t Tmf(2), \Z\big)\\
&\cong H^* \Hom_{\Z[S_3]}\big(\Z[S_3]^{\otimes(h+1)},\Hom_{\Z}(\pi_t Tmf(2),\Z)\big)\\
&\cong H^h(S_3,\pi_{-t}\IZ Tmf(2)).
\end{align*}
Indeed, the $ E_2 $-pages assemble in the following diagram
\[\xymatrix{
\Ext^v_{\Z}\big(H_{h}(S_3,\pi_t Tmf(2)), \Z\big) \ar@2[d]_-*+[o][F-]{A} \ar@2[r]^-*+[o][F-]{B} &\Ext^v_{\Z}\big(\pi_{t+h}Tmf(2)_{hS_3},\Z\big)\ar@2[d]^-*+[o][F-]{D}\\
H^{h+v}(S_3,\pi_{-t}\IZ Tmf(2))  \ar@2[r]^-*+[o][F-]{C} & \pi_{-t-h-v} \IZ Tmf(2)_{hS_3}.
}\]
The spectral sequence $ A $ is the dual module group cohomology spectral sequence \eqref{ss:dualcohomology}, and it collapses, whereas $ D $ is the Anderson duality spectral sequence \eqref{ss:AndersonDual} which likewise collapses at $ E_2 $. The spectral sequence $ B $ is dual to the homotopy orbit spectral sequence $ \check E_* $ \eqref{ss:htpyorbit}, which we have completely described in Figure \ref{fig:htpyorbit}. Now \cite[Proposition 1.3.2]{\Deligne} tells us that the differentials in $ B $ are compatible with the filtration giving $ C $ if and only if $ A $ collapses, which holds in our case. Consequently, $ B $ and $ C $ are isomorphic.	

In conclusion, we read off the differentials from the homotopy orbit spectral sequence (Figure \ref{fig:htpyorbit}). There are only non-zero $ d_5 $ and $ d_9 $. For example, the generators in degrees $ (6,3) $ and $ (2,7) $ support a $ d_5 $ as the corresponding elements in \eqref{ss:htpyorbit} are hit by a differential $ d_5 $. Similarly, the generator in $ (1,72) $ supports a $ d_9 $, as it corresponds to an element hit by a $ d_9 $. There is no possibility for any other differentials, and the chart is isomorphic to a shift by $ 21 $ of the one in Figure \ref{fig:HtpyFixed}.

By now we have an abstract isomorphism of the homotopy groups of $ \IZ Tmf $ and $ \Sigma^{21}Tmf $, as $ \pi_*Tmf $-modules. As in Theorem \ref{thm:andersontmf2}, we build a map realizing this isomorphism by specifying the dualizing class and then extending using the $ Tmf $-module structure on $ \IZ Tmf $.
\end{proof}

As a corollary, we recover \cite[Proposition 2.4.1]{\Behrens}.
\begin{cor}
At odd primes, the Gross-Hopkins dual of $ L_{K(2)}Tmf $ is $ \Sigma^{22}L_{K(2)}Tmf $.
\end{cor}
\begin{proof}
The spectrum $ Tmf $ is $ E(2) $-local, hence we can compute the Gross-Hopkins dual $ I_2Tmf $ as $ \Sigma L_{K(2)}\IZ Tmf $ by \eqref{AndersonGrossHopkins}.
\end{proof}

\bibliographystyle{amsalpha}
\providecommand{\bysame}{\leavevmode\hbox to3em{\hrulefill}\thinspace}
\providecommand{\MR}{\relax\ifhmode\unskip\space\fi MR }
\providecommand{\MRhref}[2]{%
  \href{http://www.ams.org/mathscinet-getitem?mr=#1}{#2}
}
\providecommand{\href}[2]{#2}


\end{document}